\documentclass[12pt]{amsart}
\usepackage{amsmath,amssymb,latexsym,cancel,rotating}
\usepackage{graphicx,amssymb,mathrsfs,amsmath,color,fancyhdr,amsthm}
\usepackage[all]{xy}

\newtheorem{theorem}{Theorem}[section]
\newtheorem{lemma}[theorem]{Lemma}
\newtheorem{corollary}[theorem]{Corollary}

\newtheorem{proposition}[theorem]{Proposition}

\begin{document}

\title[Geometric realizations of Lusztig's symmetries]
{Geometric realizations of Lusztig's symmetries}

\author[Xiao]{Jie Xiao}
\address{Department of Mathematical Sciences, Tsinghua University, Beijing 100084, P. R. China}
\email{jxiao@math.tsinghua.edu.cn (J.Xiao)}
\author[Zhao]{Minghui Zhao$^\dag$}
\address{College of Science, Beijing Forestry University, Beijing 100083, P. R. China}
\email{zhaomh@bjfu.edu.cn (M.Zhao)}
\thanks{$^\dag$ Corresponding author}

\subjclass[2000]{16G20, 17B37}

\date{\today}

\keywords{Lusztig's symmetries, Standard modules, Projective resolutions}

\bibliographystyle{abbrv}

\maketitle

\begin{abstract}
In this paper, we give geometric realizations of Lusztig's symmetries.
We also give projective resolutions of a kind of standard modules.
By using the geometric realizations and the projective resolutions, we
obtain the categorification of the formulas of Lusztig's symmetries.
\end{abstract}

\tableofcontents

\section{Introduction}

Let $\mathbf{U}$ be the quantum group and $\mathbf{f}$ be the Lusztig's algebra associated with a Cartan datum. Denote by $\mathbf{U}^+$ and $\mathbf{U}^-$ the positive part and the negative part of $\mathbf{U}$ respectively. There are two well-defined $\mathbb{Q}(v)$-algebra homomorphisms ${^+}:\mathbf{f}\rightarrow\mathbf{U}$ and ${^-}:\mathbf{f}\rightarrow\mathbf{U}$ with images $\mathbf{U}^+$ and $\mathbf{U}^-$ respectively.

Lusztig introduced the canonical basis $\mathbf{B}$ of $\mathbf{f}$ in \cite{Lusztig_Canonical_bases_arising_from_quantized_enveloping_algebra,Lusztig_Quivers_perverse_sheaves_and_the_quantized_enveloping_algebras,Lusztig_Introduction_to_quantum_groups}.
Let $Q=(I,H)$ be a quiver corresponding to $\mathbf{f}$ and $\mathbf{V}$ be an $I$-graded vector space such that $\underline{\dim}\mathbf{V}=\nu\in\mathbb{N}I$.
He studied the variety $E_{\mathbf{V}}$ consisting of representations of $Q$ with dimension vector $\nu$, and a category $\mathcal{Q}_{\mathbf{V}}$ of some semisimple perverse sheaves on $E_{\mathbf{V}}$. Let $K(\mathcal{Q}_{\mathbf{V}})$ be the Grothendieck group of $\mathcal{Q}_{\mathbf{V}}$. Considering all dimension vectors, he proved that $\bigoplus_{\nu\in\mathbb{N}I}K(\mathcal{Q}_{\mathbf{V}})$ realizes $\mathbf{f}$ and the set of isomorphism classes of simple objects realizes the canonical basis $\mathbf{B}$.

Lusztig also introduced some symmetries $T_i$ on $\mathbf{U}$ for all $i\in I$ in \cite{Lusztig_Quantum_deformations_of_certain_simple_modules_over_enveloping_algebras,Lusztig_Quantum_groups_at_roots_of_1}.
Note that $T_i(\mathbf{U}^+)$ is not contained in $\mathbf{U}^+$. Hence, Lusztig introduced two subalgebras ${_i\mathbf{f}}$ and ${^i\mathbf{f}}$ of $\mathbf{f}$ for any $i\in I$, where ${_i\mathbf{f}}=\{x\in\mathbf{f}\,\,|\,\,T_i(x^+)\in\mathbf{U}^+\}$ and ${^i\mathbf{f}}=\{x\in\mathbf{f}\,\,|\,\,T^{-1}_i(x^+)\in\mathbf{U}^+\}$.
Let $T_i:{_i\mathbf{f}}\rightarrow{^i\mathbf{f}}$ be the unique map satisfying
$T_i(x^+)=T_i(x)^+$.
The algebra $\mathbf{f}$ has the following direct sum decompositions
$\mathbf{f}={_i\mathbf{f}}\bigoplus\theta_i\mathbf{f}={^i\mathbf{f}}\bigoplus\mathbf{f}\theta_i$.
Denote by $_i\pi:\mathbf{f}\rightarrow{_i\mathbf{f}}$ and $^i\pi:\mathbf{f}\rightarrow{^i\mathbf{f}}$ the natural projections.

Associated to a finite dimensional hereditary algebra $\Lambda$, Ringel introduced the Hall algebra and the composition subalgebra $\mathcal{F}$ in \cite{Ringel_Hall_algebras_and_quantum_groups},
which gives a realization of the positive part of the quantum group $\mathbf{U}$. If we use the notations of Lusztig in \cite{Lusztig_Canonical_bases_and_Hall_algebras}, we have the canonical isomorphism between the composition subalgebra $\mathcal{F}$ and the Lusztig's algebra $\mathbf{f}$.
Via the Hall algebra approach, one can apply BGP-reflection functors to quantum groups to give precise constructions of Lusztig's symmetries (\cite{
Ringel_PBW-bases_of_quantum_groups,Lusztig_Canonical_bases_and_Hall_algebras,
Sevenhant_Van_den_Bergh_On_the_double_of_the_Hall_algebra_of_a_quiver,
Xiao_Yang_BGP-reflection_functors_and_Lusztig's_symmetries,Deng_Xiao,
Xiao_Zhao_BGP-reflection_functors_and_Lusztig's_symmetries_of_modified_quantized_enveloping_algebras}).

To a Lusztig's algebra $\mathbf{f}$,
Khovanov, Lauda (\cite{Khovanov_Lauda_A_diagrammatic_approach_to_categorification_of_quantum_groups_I}) and Rouquier (\cite{Rouquier_2-Kac-Moody_algebras}) introduced a series of algebras $\mathbf{R}_{\nu}$ respectively.
The category of finitely generated projective modules of $\mathbf{R}_{\nu}$ gives a categorification of $\mathbf{f}$ and $\mathbf{R}_{\nu}$ are called Khovanov-Lauda-Rouquier (KLR) algebras. Varagnolo, Vasserot (\cite{Varagnolo_Vasserot_Canonical_bases_and_KLR-algebras}) and
Rouquier (\cite{Rouquier_Quiver_Hecke_algebras_and_2-Lie_algebras}) realized the KLR algebra $\mathbf{R}_{\nu}$ as the extension algebra of semisimple perverse sheaves in $\mathcal{Q}_{\mathbf{V}}$ and proved that the set of indecomposable projective modules of $\mathbf{R}_{\nu}$ can categorify the canonical basis $\mathbf{B}$.

In \cite{Kato_An_algebraic_study_of_extension_algebra,Kato_PBW_bases_and_KLR_algebras}, Kato gave the categorification of the PBW-type bases of quantum groups of finite type. He constructed some modules (which are called standard modules) of the KLR algebras $\mathbf{R}_{\nu}$ and proved that there standard modules can
categorify the PBW-type basis of $\mathbf{f}$ by using the geometric realizations of $\mathbf{R}_{\nu}$ given by Varagnolo, Vasserot and Rouquier.
He proved that the length of the projective resolution of any standard module is finite, which is the categorification of the following fact:
the transition matrix between the PBW-type basis of $\mathbf{f}$ and the canonical basis $\mathbf{B}$ is triangular
with diagonal entries equal to $1$.
This result implies that the global dimensions of the KLR algebras $\mathbf{R}_{\nu}$ are also finite. In \cite{McNamara_Finite_dimensional_representations_of_Khovanov-Lauda-Rouquier_algebras_I,
Brundan_Kleshchev_McNamara_Homological_properties_of_finite-type_Khovanov-Lauda-Rouquier_algebras}, Brundan, Kleshchev and McNamara proved the same result by using an algebraic method.

Let $i\in I$ be a sink (resp. source) of $Q$. Similarly to the geometric realization of $\mathbf{f}$, consider a subvariety
${_i{E_{\mathbf{V}}}}$ (resp. ${^i{E_{\mathbf{V}}}}$) of $E_{\mathbf{V}}$ and a category ${_i\mathcal{Q}}_{\mathbf{V}}$ (resp. ${^i\mathcal{Q}}_{\mathbf{V}}$) of some semisimple perverse sheaves on ${_i{E_{\mathbf{V}}}}$ (resp. ${^i{E_{\mathbf{V}}}}$). In Section \ref{subsection:3.2}, we verify that $\bigoplus_{\nu\in\mathbb{N}I}K({_i\mathcal{Q}}_{\mathbf{V}})$ (resp. $\bigoplus_{\nu\in\mathbb{N}I}K({^i\mathcal{Q}}_{\mathbf{V}})$) realizes
${_i\mathbf{f}}$ (resp. ${^i\mathbf{f}}$).

Let $i\in I$ be a sink of $Q$. Let $Q'=\sigma_iQ$ be the quiver by reversing the directions
of all arrows in $Q$ containing $i$. Hence, $i$ is a source of $Q'$. Consider two $I$-graded vector spaces $\mathbf{V}$ and $\mathbf{V}'$ such that $\underline{\dim}\mathbf{V}'=s_i(\underline{\dim}\mathbf{V})$.
In the case of finite type, Kato introduced an equivalence $\tilde{\omega}_i:{_i\mathcal{Q}}_{\mathbf{V},Q}\rightarrow{^i\mathcal{Q}_{\mathbf{V}',Q'}}$ and studied the properties of this equivalence in \cite{Kato_An_algebraic_study_of_extension_algebra,Kato_PBW_bases_and_KLR_algebras}. In this paper, we generalize his construction to all cases and prove that the map induced by $\tilde{\omega}_i$ realizes the Lusztig's symmetry $T_i:{_i\mathbf{f}}\rightarrow{^i\mathbf{f}}$. For the proof of the result, we shall study the relations between the map induced by $\tilde{\omega}_i$ and the Hall algebra approach to $T_i$ in \cite{Lusztig_Canonical_bases_and_Hall_algebras}.

In \cite{Lusztig_Braid_group_action_and_canonical_bases}, Lusztig showed that Lusztig's symmetries and canonical bases are compatible.
Let ${_i\mathbf{B}}={_i\pi}(\mathbf{B})$, which
is a $\mathbb{Q}(v)$-basis of ${_i\mathbf{f}}$.
Similarly, ${^i\mathbf{B}}={^i\pi}(\mathbf{B})$ is a $\mathbb{Q}(v)$-basis of ${^i\mathbf{f}}$.
%The canonical basis $\mathbf{B}$ of $\mathbf{f}$ induces canonical bases ${_i\mathbf{B}}$ of ${_i\mathbf{f}}$ and ${^i\mathbf{B}}$ of ${^i\mathbf{f}}$ for any $i\in I$ respectively.
Lusztig proved that $T_i:{_i\mathbf{f}}\rightarrow{^i\mathbf{f}}$ maps any element of ${_i\mathbf{B}}$ to an element of ${^i\mathbf{B}}$.

For any simple perverse sheaf $\mathcal{L}$ in $\mathcal{Q}_{\mathbf{V},Q}$, the restriction ${_i\mathcal{L}}=j^{\ast}_{\mathbf{V}}(\mathcal{L})$ on ${_i{E_{\mathbf{V},Q}}}$ is also a simple perverse sheaf and belongs to ${_i\mathcal{Q}}_{\mathbf{V},Q}$, where $j_{\mathbf{V}}:{_iE}_{\mathbf{V},Q}\rightarrow E_{\mathbf{V},Q}$ is the canonical embedding. Let ${^i\mathcal{L}}=\tilde{\omega}_i({_i\mathcal{L}})\in{^i\mathcal{Q}_{\mathbf{V}',Q'}}$. The simple perverse sheaf ${^i\mathcal{L}}$ can be wrote as
${^i\mathcal{L}}=j^{\ast}_{\mathbf{V}'}(\mathcal{L}')$, where $\mathcal{L}'$ is a simple perverse sheaf in $\mathcal{Q}_{\mathbf{V}',Q'}$ and  $j_{\mathbf{V}'}:{^iE}_{\mathbf{V}',Q'}\rightarrow E_{\mathbf{V}',Q'}$ is the canonical embedding. Since the map induced by $\tilde{\omega}_i$ realizes
$T_i:{_i\mathbf{f}}\rightarrow{^i\mathbf{f}}$, this result gives a geometric interpretation of Lusztig's result in \cite{Lusztig_Braid_group_action_and_canonical_bases}.

For any $m\leq-a_{ij}$, let
\begin{displaymath}
f(i,j;m)=\sum_{r+s=m}(-1)^rv^{-r(-a_{ij}-m+1)}\theta_i^{(r)}\theta_j\theta_i^{(s)}\in\mathbf{f},
\end{displaymath}
and
\begin{displaymath}
f'(i,j;m)=\sum_{r+s=m}(-1)^rv^{-r(-a_{ij}-m+1)}\theta_i^{(s)}\theta_j\theta_i^{(r)}\in\mathbf{f}.
\end{displaymath}
In \cite{Lusztig_Introduction_to_quantum_groups}, Lusztig proved that $T_i(f(i,j;m))=f'(i,j;m')$, where $m'=-a_{ij}-m$.
The following formula
$$T_i(E_j)=\sum_{r+s=-a_{ij}}(-1)^rv^{-r}E_i^{(s)}E_jE_i^{(r)}$$
is a special case of $T_i(f(i,j;m))=f'(i,j;m')$.
In this paper, our main result is the categorification of these formulas.
Consider an $I$-graded vector space $\mathbf{V}$ such that $\underline{\dim}\mathbf{V}=\nu=mi+j$.
Let $\mathcal{D}_{G_{\mathbf{V}}}(E_{\mathbf{V}})$ be the bounded $G_{\mathbf{V}}$-equivariant derived category of complexes of $l$-adic sheaves on $E_{\mathbf{V}}$.
We construct a series of distinguished triangles in $\mathcal{D}_{G_{\mathbf{V}}}(E_{\mathbf{V}})$,
which represent the constant sheaf $\mathbf{1}_{_iE_{\mathbf{V}}}$ in terms of some semisimple perverse sheaves $I_p\in\mathcal{D}_{G_{\mathbf{V}}}(E_{\mathbf{V}})$ geometrically.
Note that, $\mathbf{1}_{_iE_{\mathbf{V}}}$ corresponds to a standard module $K_{m}$ of the KLR algebra $\mathbf{R}_{\nu}$ and $I_p$ correspond to projective modules of $\mathbf{R}_{\nu}$.
This result means that we find projective resolutions of the standard modules $K_{m}$.
Consider two $I$-graded vector spaces $\mathbf{V}$ and $\mathbf{V}'$ such that $\underline{\dim}\mathbf{V}=mi+j$ and $\underline{\dim}\mathbf{V}'=s_i(\underline{\dim}\mathbf{V})=m'i+j$.
Applying to the Grothendieck group, $\mathbf{1}_{_iE_{\mathbf{V},Q}}$ (resp. $\mathbf{1}_{^iE_{\mathbf{V}',Q'}}$) corresponds to $f(i,j;m)$ (resp. $f'(i,j;m')$).
The property of BGP-reflection functors implies $\tilde{\omega}_i(v^{-mN}\mathbf{1}_{_iE_{\mathbf{V},Q}})=v^{-m'N}\mathbf{1}_{^iE_{\mathbf{V}',Q'}}$, therefore  $T_i(f(i,j;m))=f'(i,j;m')$.

In Example D of \cite{Kato_PBW_bases_and_KLR_algebras}, Kato constructed a short exact sequence
\begin{displaymath}
\xymatrix{
0\ar[r]&{P_1\ast P_2[2]}\ar[r]&{{P_2\ast P_1}}\ar[r]&Q_{12}\ar[r]&0
}
\end{displaymath}
which coincides with the projection resolution in our main result in the case of finite type.
In Theorem 4.10 of \cite{Brundan_Kleshchev_McNamara_Homological_properties_of_finite-type_Khovanov-Lauda-Rouquier_algebras}, Brundan, Kleshchev and McNamara constructed a shout exact sequence of standard modules
\begin{displaymath}
\xymatrix{
0\ar[r]&v^{-\beta\cdot\gamma}{\Delta(\beta)\circ\Delta(\gamma)}\ar[r]&{{\Delta(\gamma)\circ\Delta(\beta)}}\ar[r]&[p_{\beta,\gamma}+1]\Delta(\alpha)\ar[r]&0.
}
\end{displaymath}
In the case of finite type, the projection resolution in our main result is a special case of the shout exact sequence above
where $\alpha=\alpha_i+\alpha_j$.

\section{Quantum groups and Lusztig's symmetries}

\subsection{Quantum groups}\label{subsection:2.1}

Let $I$ be a finite index set with $|I|=n$ and $A=(a_{ij})_{i,j\in I}$ be a generalized Cartan matrix.
Let $(A,\Pi,\Pi^{\vee},P,P^{\vee})$ be a Cartan datum associated with
$A$, where
\begin{enumerate}
  \item[(1)]$\Pi=\{\alpha_i\,\,|\,\,i\in I\}$ is the set of simple roots;
  \item[(2)]$\Pi^{\vee}=\{h_i\,\,|\,\,i\in I\}$ is the set of simple coroots;
  \item[(3)]$P$ is the weight lattice;
  \item[(4)]$P^{\vee}$ is the dual weight lattice.
\end{enumerate}
In this paper, we always assume that the generalized Cartan matrix $A$ is symmetric.
Fix an indeterminate $v$. For any $n\in\mathbb{Z}$, set
$[n]_{v}=\frac{v^n-v^{-n}}{v-v^{-1}}\in\mathbb{Q}(v)$.
Let $[0]_{v}!=1$ and $[n]_{v}!=[n]_{v}[n-1]_{v}\cdots[1]_{v}$ for any $n\in\mathbb{Z}_{>0}$.

The quantum group $\mathbf{U}$ associated with a Cartan datum $(A,\Pi,\Pi^{\vee},P,P^{\vee})$ is an associative algebra over $\mathbb{Q}(v)$ with unit element $\mathbf{1}$, generated by the elements $E_i$, $F_i(i\in I)$ and $K_{\mu}(\mu\in P^{\vee})$ subject to the following relations
\begin{displaymath}
K_{0}=\mathbf{1},\,\,\,\, K_{\mu}K_{\mu'}=K_{\mu+\mu'}\,\,\,\textrm{for all $\mu,\mu'\in P^{\vee}$};
\end{displaymath}
\begin{displaymath}
K_{\mu}E_{i}K_{-\mu}=v^{\alpha_i(\mu)}E_i\,\,\,\textrm{for all $i\in I$, $\mu\in P^{\vee}$};
\end{displaymath}
\begin{displaymath}
K_{\mu}F_{i}K_{-\mu}=v^{-\alpha_i(\mu)}F_i\,\,\,\textrm{for all $i\in I$, $\mu\in P^{\vee}$};
\end{displaymath}
\begin{displaymath}
E_iF_j-F_jE_i=\delta_{ij}\frac{K_{i}-K_{-i}}{v-v^{-1}}\,\,\,\textrm{for all $i,j\in I$};
\end{displaymath}
\begin{displaymath}
\sum_{k=0}^{1-a_{ij}}(-1)^{k}E_i^{(k)}E_jE_i^{(1-a_{ij}-k)}=0\,\,\,\textrm{for all $i\not=j\in I$};
\end{displaymath}
\begin{displaymath}
\sum_{k=0}^{1-a_{ij}}(-1)^{k}F_i^{(k)}F_jF_i^{(1-a_{ij}-k)}=0\,\,\,\textrm{for all $i\not=j\in I$}.
\end{displaymath}
Here, $K_i=K_{h_i}$ and $E_i^{(n)}=E_i^n/[n]_{v}!$, $F_i^{(n)}=F_i^n/[n]_{v}!$.

Let $\mathbf{U}^+$ (resp. $\mathbf{U}^-$) be the subalgebra of $\mathbf{U}$ generated by $E_i$ (resp. $F_i$) for all $i\in I$, and $\mathbf{U}^{0}$ be the subalgebra of $\mathbf{U}$ generated by $K_{\mu}$ for all $\mu\in P^{\vee}$. The quantum group $\mathbf{U}$ has the following triangular decomposition
\begin{displaymath}
\mathbf{U}\cong {\mathbf{U}^-}\otimes{\mathbf{U}^{0}}\otimes{\mathbf{U}^{+}}.
\end{displaymath}

Let $\mathbf{f}$ be the associative algebra defined by Lusztig in \cite{Lusztig_Introduction_to_quantum_groups}. The algebra $\mathbf{f}$ is generated by $\theta_i(i\in I)$ subject to the following relations
\begin{displaymath}
\sum_{k=0}^{1-a_{ij}}(-1)^{k}\theta_i^{(k)}\theta_j\theta_i^{(1-a_{ij}-k)}=0\,\,\,\textrm{for all $i\not=j\in I$},
\end{displaymath}
where
$\theta_i^{(n)}=\theta_i^n/[n]_{v}!$.

There are two well-defined $\mathbb{Q}(v)$-algebra homomorphisms ${^+}:\mathbf{f}\rightarrow\mathbf{U}$ and ${^-}:\mathbf{f}\rightarrow\mathbf{U}$ satisfying $E_i=\theta_i^+$ and $F_i=\theta_i^-$ for all $i\in I$. The images of ${^+}$ and ${^-}$ are $\mathbf{U}^+$ and $\mathbf{U}^-$ respectively.

\subsection{Lusztig's symmetries}\label{subsection:2.2}

Corresponding to $i\in I$, Lusztig introduced the Lusztig's symmetry $T_i:\mathbf{U}\rightarrow \mathbf{U}$ (\cite{Lusztig_Quantum_deformations_of_certain_simple_modules_over_enveloping_algebras,Lusztig_Quantum_groups_at_roots_of_1,Lusztig_Introduction_to_quantum_groups}).
The formulas of $T_i$ on the generators are:
\begin{eqnarray}
&&T_i(E_i)=-F_i K_{i},\,\,\,T_i(F_i)=-K_{-i}E_i;\nonumber\\
&&T_i(E_j)=\sum_{r+s=-a_{ij}}(-1)^rv^{-r}E_i^{(s)}E_jE_i^{(r)}\,\,\,\textrm{for $i\neq j\in I$};\label{equation:2.3.1}\\
&&T_i(F_j)=\sum_{r+s=-a_{ij}}(-1)^rv^{r}F_i^{(r)}F_jF_i^{(s)}\,\,\,\textrm{for $i\neq j\in I$};\label{equation:2.3.2}\\
&&T_i(K_{\mu})=K_{\mu-\alpha_{i}(\mu)h_i}.\nonumber
\end{eqnarray}
Lusztig introduced two subalgebras ${_i\mathbf{f}}$ and ${^i\mathbf{f}}$ of $\mathbf{f}$.
For any $j\in I$, $i\neq j$, $m\in\mathbb{N}$, define
\begin{displaymath}
f(i,j;m)=\sum_{r+s=m}(-1)^rv^{-r(-a_{ij}-m+1)}\theta_i^{(r)}\theta_j\theta_i^{(s)}\in\mathbf{f},
\end{displaymath}
and
\begin{displaymath}
f'(i,j;m)=\sum_{r+s=m}(-1)^rv^{-r(-a_{ij}-m+1)}\theta_i^{(s)}\theta_j\theta_i^{(r)}\in\mathbf{f}.
\end{displaymath}
The subalgebras ${_i\mathbf{f}}$ and ${^i\mathbf{f}}$ are generated by $f(i,j;m)$ and $f'(i,j;m)$ respectively.

Note that ${_i\mathbf{f}}=\{x\in\mathbf{f}\,\,|\,\,T_i(x^+)\in\mathbf{U}^+\}$ and ${^i\mathbf{f}}=\{x\in\mathbf{f}\,\,|\,\,T^{-1}_i(x^+)\in\mathbf{U}^+\}$ (\cite{Lusztig_Introduction_to_quantum_groups}).
Hence there exists a unique $T_i:{_i\mathbf{f}}\rightarrow{^i\mathbf{f}}$ such that
$T_i(x^+)=T_i(x)^+$. Lusztig also showed that $\mathbf{f}$ has the following direct sum decompositions
$\mathbf{f}={_i\mathbf{f}}\bigoplus\theta_i\mathbf{f}={^i\mathbf{f}}\bigoplus\mathbf{f}\theta_i$.
Denote by $_i\pi:\mathbf{f}\rightarrow{_i\mathbf{f}}$ and $^i\pi:\mathbf{f}\rightarrow{^i\mathbf{f}}$ the natural projections.

Lusztig also proved the following formulas.

\begin{proposition}[\cite{Lusztig_Introduction_to_quantum_groups}]\label{proposition:5.1}
For any $-a_{ij}\geq m\in\mathbb{N}$, $T_i(f(i,j;m))=f'(i,j;-a_{ij}-m)$.
\end{proposition}

The formulas (\ref{equation:2.3.1}) and (\ref{equation:2.3.2}) are two special cases of Proposition \ref{proposition:5.1}.

\section{Geometric realizations}\label{section:3}

\subsection{Geometric realization and canonical basis of $\mathbf{f}$}\label{subsection:3.1}

In this subsection, we shall review the geometric realization of $\mathbf{f}$ introduced by Lusztig (\cite{Lusztig_Canonical_bases_arising_from_quantized_enveloping_algebra,Lusztig_Quivers_perverse_sheaves_and_the_quantized_enveloping_algebras,Lusztig_Introduction_to_quantum_groups}).

A quiver $Q=(I,H,s,t)$ consists of a vertex set $I$, an arrow set $H$, and two maps $s,t:H\rightarrow I$ such that an arrow $\rho\in H$ starts at $s(\rho)$ and terminates at $t(\rho)$. Let $h_{ij}=\#\{i\rightarrow j\}$, $a_{ij}=h_{ij}+h_{ji}$ and $\mathbf{f}$ be the Lusztig's algebra corresponding to $A=(a_{ij})$.
Let $p$ be a prime and $q$ be a power of $p$. Denote by $\mathbb{F}_q$ the finite field with $q$ elements and $\mathbb{K}=\overline{\mathbb{F}}_q$.

For a finite dimensional $I$-graded $\mathbb{K}$-vector space $\mathbf{V}=\bigoplus_{i\in I}V_i$, define
$$E_\mathbf{V}=\bigoplus_{\rho\in H}\textrm{Hom}_{\mathbb{K}}(V_{s(\rho)},V_{t(\rho)}).$$
The dimension vector of $\mathbf{V}$ is defined as $\underline{\dim}\mathbf{V}=\sum_{i\in I}(\dim_{\mathbb{K}}V_{i})i\in\mathbb{N}I$. The algebraic group $G_{\mathbf{V}}=\prod_{i\in I}GL_{\mathbb{K}}(V_i)$ acts on $E_\mathbf{V}$ naturally.

Fix a nonzero element $\nu\in\mathbb{N}I$. Let
$$Y_{\nu}=\{\mathbf{y}=(\mathbf{i},\mathbf{a})\,\,|\,\,\sum_{l=1}^{k}a_li_l=\nu\},$$
where $\mathbf{i}=(i_1,i_2,\ldots,i_k),\,\,i_l\in I$, $\mathbf{a}=(a_1,a_2,\ldots,a_k),\,\,a_l\in\mathbb{N}$, and $$I^{\nu}=\{\mathbf{i}=(i_1,i_2,\ldots,i_k)\,\,|\,\,\sum_{l=1}^{k}i_l=\nu\}.$$
Fix a finite dimensional $I$-graded $\mathbb{K}$-vector space $\mathbf{V}$ such that $\underline{\dim}\mathbf{V}=\nu$.
For any element $\mathbf{y}=(\mathbf{i},\mathbf{a})$,
a flag of type $\mathbf{y}$ in $\mathbf{V}$ is a sequence
$$\phi=(\mathbf{V}=\mathbf{V}^k\supset\mathbf{V}^{k-1}\supset\dots\supset\mathbf{V}^0=0)$$
of $I$-graded $\mathbb{K}$-vector spaces such that $\underline{\dim}\mathbf{V}^l/\mathbf{V}^{l-1}=a_li_l$. Let $F_{\mathbf{y}}$ be the variety of all flags of type $\mathbf{y}$ in $\mathbf{V}$. For any $x\in E_{\mathbf{V}}$, a flag $\phi$ is called $x$-stable if $x_{\rho}(V^l_{s(\rho)})\subset{V}^l_{t(\rho)}$ for all $l$ and all $\rho\in H$. Let
$$\tilde{F}_{\mathbf{y}}=\{(x,\phi)\in E_{\mathbf{V}}\times F_{\mathbf{y}}\,\,|\,\,\textrm{$\phi$ is $x$-stable}\}$$
and $\pi_{\mathbf{y}}:\tilde{F}_{\mathbf{y}}\rightarrow E_{\mathbf{V}}$
be the projection to $E_{\mathbf{V}}$.

Let
$\bar{\mathbb{Q}}_{l}$ be the $l$-adic field and
$\mathcal{D}_{G_{\mathbf{V}}}(E_{\mathbf{V}})$ be the bounded $G_{\mathbf{V}}$-equivariant derived category of complexes of $l$-adic sheaves on $E_{\mathbf{V}}$. For each $\mathbf{y}\in{Y}_{\nu}$, $\mathcal{L}_{\mathbf{y}}=(\pi_{\mathbf{y}})_!(\mathbf{1}_{\tilde{F}_{\mathbf{y}}})[d_{\mathbf{y}}]\in\mathcal{D}_{G_{\mathbf{V}}}(E_{\mathbf{V}})$ is a semisimple perverse sheaf, where $d_{\mathbf{y}}=\dim\tilde{F}_{\mathbf{y}}$.
Let $\mathcal{P}_{\mathbf{V}}$ be the set of isomorphism classes of simple perverse sheaves $\mathcal{L}$ on $E_{\mathbf{V}}$ such that $\mathcal{L}[r]$ appears as a direct summand of $\mathcal{L}_{\mathbf{i}}$ for some $\mathbf{i}\in I^{\nu}$ and $r\in\mathbb{Z}$. Let $\mathcal{Q}_{\mathbf{V}}$ be the full subcategory of $\mathcal{D}_{G_{\mathbf{V}}}(E_{\mathbf{V}})$ consisting of all complexes which are isomorphic to finite direct sums of complexes in the set
$\{\mathcal{L}[r]\,\,|\,\,\mathcal{L}\in\mathcal{P}_{\mathbf{V}},r\in\mathbb{Z}\}$.

Let $K(\mathcal{Q}_{\mathbf{V}})$ be the Grothendieck group of $\mathcal{Q}_{\mathbf{V}}$.
Define $$v^{\pm}[\mathcal{L}]=[\mathcal{L}[\pm1](\pm\frac{1}{2})],$$
where $\mathcal{L}(d)$ is the Tate twist of $\mathcal{L}$. Then, $K(\mathcal{Q}_{\mathbf{V}})$ is a free $\mathcal{A}$-module, where $\mathcal{A}=\mathbb{Z}[v,v^{-1}]$.
Define $$K(\mathcal{Q})=\bigoplus_{\nu\in\mathbb{N}I}K(\mathcal{Q}_{\mathbf{V}}).$$

For $\nu,\nu',\nu''\in\mathbb{N}I$ such that $\nu=\nu'+\nu''$ and three $I$-graded $\mathbb{K}$-vector spaces $\mathbf{V}$, $\mathbf{V}'$, $\mathbf{V}''$ such that $\underline{\dim}\mathbf{V}=\nu$, $\underline{\dim}\mathbf{V}'=\nu'$, $\underline{\dim}\mathbf{V}''=\nu''$, Lusztig constructed a functor
$$\ast:\mathcal{Q}_{\mathbf{V}'}\times\mathcal{Q}_{\mathbf{V}''}\rightarrow\mathcal{Q}_{\mathbf{V}}.$$
This functor induces an associative $\mathcal{A}$-bilinear multiplication
\begin{eqnarray*}
\circledast:K(\mathcal{Q}_{\mathbf{V}'})\times K(\mathcal{Q}_{\mathbf{V}''})&\rightarrow&K(\mathcal{Q}_{\mathbf{V}})\\
([\mathcal{L}']\,,\,[\mathcal{L}''])&\mapsto&[\mathcal{L}']\circledast[\mathcal{L}'']=[\mathcal{L}'\circledast\mathcal{L}'']
\end{eqnarray*}
where $\mathcal{L}'\circledast\mathcal{L}''=(\mathcal{L}'\ast\mathcal{L}'')[m_{\nu'\nu''}](\frac{m_{\nu'\nu''}}{2})$
and $m_{\nu'\nu''}=\sum_{\rho\in H}\nu'_{s(\rho)}\nu''_{t(\rho)}-\sum_{i\in I}\nu'_i\nu''_i$.
Then $K(\mathcal{Q})$ becomes an associative $\mathcal{A}$-algebra and the set $\{[\mathcal{L}]\,\,|\,\,\mathcal{L}\in\mathcal{P}_{\mathbf{V}}\}$ is a basis of $K(\mathcal{Q}_{\mathbf{V}})$.

\begin{theorem}[\cite{Lusztig_Quivers_perverse_sheaves_and_the_quantized_enveloping_algebras}]\label{theorem:3.1}
There is a unique $\mathcal{A}$-algebra isomorphism
$$\lambda_{\mathcal{A}}:K(\mathcal{Q})\rightarrow\mathbf{f}_{\mathcal{A}}$$
such that $\lambda_{\mathcal{A}}(\mathcal{L}_{\mathbf{y}})=\theta_{\mathbf{y}}$ for all $\mathbf{y}\in Y_{\nu}$, where $\theta_{\mathbf{y}}=\theta_{i_1}^{(a_1)}\theta_{i_2}^{(a_2)}\cdots\theta_{i_k}^{(a_k)}$ and $\mathbf{f}_{\mathcal{A}}$ is the integral form of $\mathbf{f}$.
\end{theorem}

Let $\mathbf{B}_{\nu}=\{\mathbf{b}_{\mathcal{L}}=\lambda_{\mathcal{A}}([\mathcal{L}])\,\,|\,\,\mathcal{L}\in\mathcal{P}_{\mathbf{V}}\}$
and $\mathbf{B}=\sqcup_{\nu\in\mathbb{N}I}\mathbf{B}_{\nu}$.
Then $\mathbf{B}$ is the canonical basis of $\mathbf{f}$ introduced by Lusztig in \cite{Lusztig_Canonical_bases_arising_from_quantized_enveloping_algebra,Lusztig_Quivers_perverse_sheaves_and_the_quantized_enveloping_algebras}.

\subsection{Geometric realizations of ${_i\mathbf{f}}$ and ${^i\mathbf{f}}$}\label{subsection:3.2}
\label{subsec:i_f_geometric}

Assume that $i\in I$ is a sink.
Let $\mathbf{V}$ be a finite dimensional $I$-graded $\mathbb{K}$-vector space  such that $\underline{\dim}\mathbf{V}=\nu$. Consider a subvariety ${_iE}_\mathbf{V}$ of $E_\mathbf{V}$
$$
{_iE}_{\mathbf{V}}=\{x\in E_{\mathbf{V}}\,\,|\,\,\bigoplus_{h\in H,t(h)=i}x_{h}:\bigoplus_{h\in H,t(h)=i}V_{s(h)}\rightarrow V_i \,\,\,\textrm{is surjective}\}.
$$
Let $j_{\mathbf{V}}:{_iE}_{\mathbf{V}}\rightarrow E_{\mathbf{V}}$ be the canonical embedding.
For any $\mathbf{y}=(\mathbf{i},\mathbf{a})\in{Y}_{\nu}$, let
\begin{displaymath}
{_i\tilde{F}_{\mathbf{y}}}=\{(x,\phi)\in {_iE}_{\mathbf{V}}\times F_{\mathbf{y}}\,\,|\,\,\textrm{$\phi$ is $x$-stable}\}
\end{displaymath}
and
${_i\pi}_{\mathbf{y}}:{_i\tilde{F}}_{\mathbf{y}}\rightarrow {_iE}_{\mathbf{V}}$
be the projection to ${_iE}_{\mathbf{V}}$.

For any $\mathbf{y}\in{Y}_{\nu}$,
$_i\mathcal{L}_{\mathbf{y}}=({_i\pi}_{\mathbf{y}})_!(\mathbf{1}_{{_i\tilde{F}}_{\mathbf{y}}})[d_{\mathbf{y}}]\in\mathcal{D}_{G_{\mathbf{V}}}(_iE_{\mathbf{V}})$ is a semisimple perverse sheaf.
Let ${_i\mathcal{P}}_{\mathbf{V}}$ be the set of isomorphism classes of simple perverse sheaves $\mathcal{L}$ on ${_iE}_{\mathbf{V}}$ such that $\mathcal{L}[r]$ appears as a direct summand of ${_i\mathcal{L}}_{\mathbf{i}}$ for some $\mathbf{i}\in I^{\nu}$ and $r\in\mathbb{Z}$. Let ${_i\mathcal{Q}}_{\mathbf{V}}$ be the full subcategory of $\mathcal{D}_{G_{\mathbf{V}}}(_iE_{\mathbf{V}})$ consisting of all complexes which are isomorphic to finite direct sums of complexes in the set $\{\mathcal{L}[r]\,\,|\,\,\mathcal{L}\in{_i\mathcal{P}}_{\mathbf{V}},r\in\mathbb{Z}\}$.

Let $K(_i\mathcal{Q}_{\mathbf{V}})$ be the Grothendieck group of $_i\mathcal{Q}_{\mathbf{V}}$ and
$$K(_i\mathcal{Q})=\bigoplus_{[V]}K(_i\mathcal{Q}_{\mathbf{V}}).$$
Naturally, we have two functors
$j_{\mathbf{V}!}:\mathcal{D}_{G_{\mathbf{V}}}(_iE_{\mathbf{V}})\rightarrow \mathcal{D}_{G_{\mathbf{V}}}(E_{\mathbf{V}})$ and
$j^{\ast}_{\mathbf{V}}:\mathcal{D}_{G_{\mathbf{V}}}(E_{\mathbf{V}})\rightarrow\mathcal{D}_{G_{\mathbf{V}}}(_iE_{\mathbf{V}})$.

For any $\mathbf{y}\in{Y}_{\nu}$, we have the following fiber product
$$\xymatrix{_i\tilde{F}_{\mathbf{y}}\ar[r]^-{\tilde{j}_\mathbf{V}}\ar[d]^-{_i\pi_{y}}&\tilde{F}_{\mathbf{y}}
\ar[d]^-{\pi_{y}}\\_iE_{\mathbf{V}}\ar[r]^-{j_\mathbf{V}}&E_{\mathbf{V}}}$$
So
\begin{equation}\label{equation:3.2.2}
j^{\ast}_\mathbf{V}\mathcal{L}_{\mathbf{y}}=j^{\ast}_\mathbf{V}(\pi_{y})_!(\mathbf{1}_{{\tilde{F}}_{\mathbf{y}}})[d_{\mathbf{y}}]
=({_i\pi}_{y})_!\tilde{j}^{\ast}_\mathbf{V}(\mathbf{1}_{{\tilde{F}}_{\mathbf{y}}})[d_{\mathbf{y}}]
=({_i\pi}_{y})_!(\mathbf{1}_{{_i\tilde{F}}_{\mathbf{y}}})[d_{\mathbf{y}}]={_i\mathcal{L}}_{\mathbf{y}}.
\end{equation}
That is $j^{\ast}_{\mathbf{V}}(\mathcal{Q}_{\mathbf{V}})={_i\mathcal{Q}}_{\mathbf{V}}$. Hence $j^{\ast}_{\mathbf{V}}:\mathcal{Q}_{\mathbf{V}}\rightarrow{_i\mathcal{Q}}_{\mathbf{V}}$ and $j^{\ast}:K(\mathcal{Q})\rightarrow K({_i\mathcal{Q}})$ can be defined.

Consider the following diagram
\begin{equation}\label{equation:3.2.1}
\xymatrix{
0\ar[r] & \theta_i\mathbf{f}_{\mathcal{A}}\ar[r]^-{i}& \mathbf{f}_{\mathcal{A}}\ar[r]^-{_i\pi_{\mathcal{A}}}\ar[d]^-{\lambda'_{\mathcal{A}}}    & {_i\mathbf{f}}_{\mathcal{A}}\ar[r]\ar@{.>}[d]    & 0 \\
        &                          & K(\mathcal{Q})\ar[r]^-{j^{\ast}}& K(_i\mathcal{Q})\ar[r]& 0
}
\end{equation}
where $\lambda'_{\mathcal{A}}$ is the inverse of $\lambda_{\mathcal{A}}$.
Since $j^{\ast}\circ{\lambda'_{\mathcal{A}}}\circ{i}=0$,
there exists a map ${_i\lambda'_{\mathcal{A}}}:{_i\mathbf{f}}_{\mathcal{A}}\rightarrow K(_i\mathcal{Q})$ such that the above diagram (\ref{equation:3.2.1}) commutes.

\begin{proposition}\label{theorem:3.5}
The map ${_i\lambda'_{\mathcal{A}}}:{_i\mathbf{f}}_{\mathcal{A}}\rightarrow K(_i\mathcal{Q})$ is an isomorphism of $\mathcal{A}$-algebras.
\end{proposition}

The proof of Proposition \ref{theorem:3.5} will be given in Section \ref{subsection:4.2}.

Assume that $i\in I$ is a source. We can give a geometric realization of ${^i\mathbf{f}}$ similarly. Consider a subvariety ${^iE}_\mathbf{V}$ of $E_\mathbf{V}$
$$
{^iE}_{\mathbf{V}}=\{x\in E_{\mathbf{V}}\,\,|\,\,\bigoplus_{h\in H,s(h)=i}x_{h}:V_i\rightarrow\bigoplus_{h\in H,s(h)=i}V_{t(h)}\,\,\,\textrm{is injective}\}.
$$
Let $j_{\mathbf{V}}:{^iE}_{\mathbf{V}}\rightarrow E_{\mathbf{V}}$ be the canonical embedding.
The definitions of ${^i\mathcal{Q}}_{\mathbf{V}}$, $K(^i\mathcal{Q}_{\mathbf{V}})$ and $K(^i\mathcal{Q})$ are similar to those of ${_i\mathcal{Q}}_{\mathbf{V}}$, $K(_i\mathcal{Q}_{\mathbf{V}})$ and $K(_i\mathcal{Q})$ respectively.
We can also define $j^{\ast}_{\mathbf{V}}:\mathcal{Q}_{\mathbf{V}}\rightarrow{^i\mathcal{Q}}_{\mathbf{V}}$, $j^{\ast}:K(\mathcal{Q})\rightarrow K({^i\mathcal{Q}})$ and ${^i\lambda'_{\mathcal{A}}}:{^i\mathbf{f}}_{\mathcal{A}}\rightarrow K(^i\mathcal{Q})$.

Similarly to Proposition \ref{theorem:3.5}, we have the following proposition.

\begin{proposition}\label{theorem:3.8}
The map ${^i\lambda'_{\mathcal{A}}}:{^i\mathbf{f}}_{\mathcal{A}}\rightarrow K(^i\mathcal{Q})$ is an isomorphism of $\mathcal{A}$-algebras.
\end{proposition}

\qed

\subsection{Geometric realization of $T_i:{_i\mathbf{f}}\rightarrow{^i\mathbf{f}}$}\label{subsection:3.4}

Assume that $i$ is a sink of $Q=(I,H,s,t)$. So $i$ is a source of $Q'=\sigma_iQ=(I,H',s,t)$, where $\sigma_iQ$ is the quiver by reversing the directions
of all arrows in $Q$ containing $i$. For any $\nu,\nu'\in\mathbb{N}I$ such that $\nu'=s_i\nu$ and $I$-graded $\mathbb{K}$-vector spaces $\mathbf{V}$,  $\mathbf{V}'$ such that $\underline{\dim}\mathbf{V}=\nu$, $\underline{\dim}\mathbf{V}'=\nu'$, consider the following correspondence (\cite{Lusztig_Canonical_bases_and_Hall_algebras,Kato_PBW_bases_and_KLR_algebras})
\begin{equation}\label{equation:3.4.2}
\xymatrix{
_iE_{\mathbf{V},Q}&Z_{\mathbf{V}\mathbf{V}'}\ar[l]_-{\alpha}\ar[r]^-{\beta}&^iE_{\mathbf{V}',Q'}},
\end{equation}
where
\begin{enumerate}
  \item[(1)]$Z_{\mathbf{V}\mathbf{V}'}$ is the subset in $E_{\mathbf{V},Q}\times E_{\mathbf{V}',Q'}$ consisting of all $(x,y)$ satisfying the following conditions
           \begin{enumerate}
                  \item[(a)]for any $h\in H$ such that $t(h)\neq i$ and $h\in H'$, $x_h=y_h$;
                  \item[(b)]the following sequence is exact
                  $$
                  \xymatrix@=15pt{0\ar[r]&}
                  \xymatrix@=50pt{V'_i\ar[r]^-{\bigoplus_{h\in H',s(h)=i}y_h}&\bigoplus_{h\in H,t(h)=i}V_{s(h)}\ar[r]^-{\bigoplus_{h\in H,t(h)=i}x_h}&V_i}
                  \xymatrix@=15pt{\ar[r]&0}
                  $$
           \end{enumerate}
  \item[(2)]$\alpha(x,y)=x$ and $\beta(x,y)=y$.
\end{enumerate}
From now on, $_iE_{\mathbf{V},Q}$ is denoted by $_iE_{\mathbf{V}}$ and $^iE_{\mathbf{V}',Q'}$ is denoted by $^iE_{\mathbf{V}'}$.
Let
$$
G_{\mathbf{V}\mathbf{V}'}=GL(V_i)\times GL(V'_i)\times\prod_{j\neq i}GL(V_j)\cong GL(V_i)\times GL(V'_i)\times\prod_{j\neq i}GL(V'_j),
$$
which acts on $Z_{\mathbf{V}\mathbf{V}'}$ naturally.

By (\ref{equation:3.4.2}), we have
\begin{equation}\label{equation:3.4.1}
\xymatrix{
\mathcal{D}_{G_{\mathbf{V}}}(_iE_{\mathbf{V}})\ar[r]^-{\alpha^{\ast}}&\mathcal{D}_{G_{\mathbf{V}\mathbf{V}'}}(Z_{\mathbf{V}\mathbf{V}'})&\mathcal{D}_{G_{\mathbf{V}'}}(^iE_{\mathbf{V}'})\ar[l]_-{\beta^{\ast}}
}.\end{equation}
Since $\alpha$ and $\beta$ are principal bundles with fibers $Aut(V'_i)$ and $Aut(V_i)$ respectively, $\alpha^{\ast}$ and $\beta^{\ast}$ are equivalences of categories by Section 2.2.5 in \cite{Bernstein_Lunts_Equivariant_sheaves_and_functors}.
Hence, for any $\mathcal{L}\in\mathcal{D}_{G_{\mathbf{V}}}(_iE_{\mathbf{V}})$£¬
there exists a unique $\mathcal{L}'\in\mathcal{D}_{G_{\mathbf{V}'}}(^iE_{\mathbf{V}'})$ such that $\alpha^{\ast}(\mathcal{L})=\beta^{\ast}(\mathcal{L}')$. Define
\begin{eqnarray*}
\tilde{\omega}_i:\mathcal{D}_{G_{\mathbf{V}}}(_iE_{\mathbf{V}})&\rightarrow&\mathcal{D}_{G_{\mathbf{V}'}}(^iE_{\mathbf{V}'})\\
\mathcal{L}&\mapsto&\mathcal{L}'[-s(\mathbf{V})](-\frac{s(\mathbf{V})}{2})
\end{eqnarray*}
where
$
s(\mathbf{V})=\dim\textrm{GL}(V_i)-\dim\textrm{GL}(V'_i).
$
Since $\alpha^{\ast}$ and $\beta^{\ast}$ are equivalences of categories, $\tilde{\omega}_i$ is also an equivalence of categories.

\begin{proposition}\label{proposition:3.9}
It holds that  $\tilde{\omega}_i({_i\mathcal{Q}}_{\mathbf{V}})={^i\mathcal{Q}}_{\mathbf{V}'}$.
\end{proposition}

The proof of Proposition \ref{proposition:3.9} will be given in Section \ref{subsection:4.4}.

Hence, we can define
$\tilde{\omega}_i:{_i\mathcal{Q}}_\mathbf{V}\rightarrow{^i\mathcal{Q}}_{\mathbf{V}'}$
and
$\tilde{\omega}_i:K({_i\mathcal{Q}})\rightarrow K({^i\mathcal{Q}})$.
We have the following theorem.

\begin{theorem}\label{theorem:3.10}
We have the following commutative diagram
$$\xymatrix{{_i\mathbf{f}}_{\mathcal{A}}\ar[r]^-{T_i}\ar[d]^-{_i\lambda'_{\mathcal{A}}}&{^i\mathbf{f}}_{\mathcal{A}}\ar[d]^-{^i\lambda'_{\mathcal{A}}}\\K({_i\mathcal{Q}})\ar[r]^-{\tilde{\omega}_i}&K({^i\mathcal{Q}})}$$
\end{theorem}

The proof of Theorem \ref{theorem:3.10} will be given in Section \ref{subsection:4.4}.

\subsection{$T_i:{_i\mathbf {f}}\rightarrow {^i\mathbf {f}}$ and canonical bases}\label{subsection:3.5}

In \cite{Lusztig_Braid_group_action_and_canonical_bases}, Lusztig showed that Lusztig's symmetries and canonical bases are compatible. In this section, we shall give a geometric interpretation of this result by using the geometric realization of $T_i$.

Let $\mathbf{B}$ be the canonical basis of $\mathbf{f}$.
Since $\theta_i\mathbf{f}$ is the kernel of ${_i\pi}:\mathbf{f}\rightarrow{_i\mathbf{f}}$ and $\mathbf{B}\cap\theta_i\mathbf{f}$ is a $\mathbb{Q}(v)$-basis of $\theta_i\mathbf{f}$,
%(\cite{Lusztig_Introduction_to_quantum_groups}),
%${_i\mathbf{B}}={_i\pi}(\mathbf{B}-\mathbf{B}\cap\theta_i\mathbf{f})$
${_i\mathbf{B}}={_i\pi}(\mathbf{B})$
is a $\mathbb{Q}(v)$-basis of ${_i\mathbf{f}}$.
Similarly, ${^i\mathbf{B}}={^i\pi}(\mathbf{B})$ is a $\mathbb{Q}(v)$-basis of ${^i\mathbf{f}}$.

Lusztig proved the following theorem.

\begin{theorem}[\cite{Lusztig_Braid_group_action_and_canonical_bases}]\label{theorem:3.11}
Lusztig's symmetry $T_i:{_i\mathbf{f}}\rightarrow{^i\mathbf{f}}$ maps any element of ${_i\mathbf{B}}$ to an element of ${^i\mathbf{B}}$.
Thus, there exists a unique bijection
$\kappa_i:\mathbf{B}-\mathbf{B}\cap\theta_i\mathbf{f}\rightarrow\mathbf{B}-\mathbf{B}\cap\mathbf{f}\theta_i$
such that $T_i(_i\pi(b))={^i\pi}(\kappa_i(b))$.
\end{theorem}

Let $i$ be a sink of a quiver $Q$. So $i$ is a source of $Q'=\sigma_iQ$.
By Theorem \ref{theorem:3.1}, Proposition \ref{theorem:3.8}, the formula (\ref{equation:3.2.2}) and the commutative diagram (\ref{equation:3.2.1}), we have \begin{equation}\label{equation:3.5.1}
{_i\mathbf{B}}=\sqcup_{\nu\in\mathbb{N}I}\{\mathbf{b}_{\mathcal{L}}=
{_i\lambda_{\mathcal{A}}}([\mathcal{L}])\,\,|\,\,\mathcal{L}\in{_i\mathcal{P}_{\mathbf{V}}},\,\,\underline{\dim}\mathbf{V}=\nu\}.
\end{equation}
Similarly, we have
\begin{equation}\label{equation:3.5.2}
{^i\mathbf{B}}=\sqcup_{ \nu'\in\mathbb{N}I}\{\mathbf{b}_{\mathcal{L}}={^i\lambda_{\mathcal{A}}}([\mathcal{L}])\,\,|\,\,\mathcal{L}\in{^i\mathcal{P}_{\mathbf{V}'}},\,\,\underline{\dim} \mathbf{V}'=\nu'\}.
\end{equation}

Fix any $\nu,\nu'\in\mathbb{N}I$ such that $\nu'=s_i\nu$ and $I$-graded $\mathbb{K}$-vector spaces $\mathbf{V}$, $\mathbf{V}'$ such that $\underline{\dim}\mathbf{V}=\nu$, $\underline{\dim}\mathbf{V}'=\nu'$.

In (\ref{equation:3.4.1}), the functors $\alpha^{\ast}$ and $\beta^{\ast}$ are equivalences of categories. Hence the functor
$$
\tilde{\omega}_i:{_i\mathcal{Q}}_\mathbf{V}\rightarrow{^i\mathcal{Q}}_{\mathbf{V}'}
$$
maps any simple perverse sheaf in ${_i\mathcal{Q}}_\mathbf{V}$ to a simple perverse sheaf in ${^i\mathcal{Q}}_{\mathbf{V}'}$. That is, $\tilde{\omega}_i({_i\mathcal{P}}_{\mathbf{V}})={^i\mathcal{P}}_{\mathbf{V}'}$. So the map
$$
\tilde{\omega}_i:K({_i\mathcal{Q}})\rightarrow K({^i\mathcal{Q}})
$$
satisfies $$\tilde{\omega}_i(\{[\mathcal{L}]\,\,|\,\,\mathcal{L}\in{_i\mathcal{P}}_{\mathbf{V}}\})=\{[\mathcal{L}]\,\,|\,\,\mathcal{L}\in{^i\mathcal{P}}_{\mathbf{V}'}\}.$$
By Theorem \ref{theorem:3.10}, (\ref{equation:3.5.1}) and (\ref{equation:3.5.2}),
%$
%T_i:{_i\mathbf{f}}\rightarrow{^i\mathbf{f}}
%$
it holds that $T_i(_i\mathbf{B})={^i\mathbf{B}}$ and we get a geometric interpretation of Theorem \ref{theorem:3.11}.

\section{Hall algebra approaches}

\subsection{Hall algebra approach to $\mathbf{f}$}\label{subsection:4.1}

In this subsection, we shall review the Hall algebra approach to $\mathbf{f}$ (\cite{Ringel_Hall_algebras_and_quantum_groups,Lusztig_Canonical_bases_and_Hall_algebras,Lin_Lusztig's_geometric,Lin_Xiao_Zhang_Representations_of_tame_quivers_and_affine_canonical_bases}).

Let $Q=( I,H,s,t)$ be a quiver. In Section \ref{subsection:3.1}, $E_\mathbf{V}$ and $G_{\mathbf{V}}$ are defined for any
$I$-graded $\mathbb{K}$-vector space $\mathbf{V}$.
Let $F^n$ be the Frobenius morphism.
The sets $E^{F^n}_\mathbf{V}$ and $G^{F^n}_{\mathbf{V}}$ consist of the $F^n$-fixed points in $E_\mathbf{V}$ and $G_{\mathbf{V}}$ respectively.

Lusztig defined $\underline{\mathcal{F}}^n_{\mathbf{V}}$ as the set of all $G^{F^n}_{\mathbf{V}}$-invariant $\bar{\mathbb{Q}}_l$-functions on $E^{F^n}_\mathbf{V}$
and we can give a multiplication on $\underline{\mathcal{F}}^n=\bigoplus_{\nu\in\mathbb{N}I}\underline{\mathcal{F}}^n_{\mathbf{V}}$
to obtain the Hall algebra. For any $i\in I$,
let $\mathbf{V}_i$ be the $I$-graded $\mathbb{K}$-vector space with dimension vector $i$ and $f_i$ be the constant function on $E^{F^n}_{\mathbf{V}_i}$ with value $1$. Denote by $\mathcal{F}^n$ the composition subalgebra of
$\underline{\mathcal{F}}^n$ generated by $f_i$
and $\mathcal{F}^n_{\mathbf{V}}=\underline{\mathcal{F}}^n_{\mathbf{V}}\cap\mathcal{F}^n$.
Let
$\mathcal{F}=\bigoplus_{\nu\in\mathbb{N}I}\mathcal{F}_{\mathbf{V}}$
be the generic form of $\mathcal{F}^n$ and $\mathcal{F}_{\mathcal{A}}$ be the integral form of $\mathcal{F}$ (\cite{Lusztig_Canonical_bases_and_Hall_algebras}).

\begin{theorem}[\cite{Ringel_Hall_algebras_and_quantum_groups,Lusztig_Canonical_bases_and_Hall_algebras}]\label{theorem:4.2}
There exists an isomorphism of $\mathcal{A}$-algebras
$$\varpi_{\mathcal{A}}:\mathbf{f}_\mathcal{A}\rightarrow \mathcal{F}_\mathcal{A}$$
such that $\varpi_{\mathcal{A}}(\theta_i)=f_i$.
\end{theorem}

For any $\mathcal{L}\in\mathcal{D}_{G_{\mathbf{V}}}(E_{\mathbf{V}})$, there is a function $\chi^n_{\mathcal{L}}:E^{F^n}_\mathbf{V}\rightarrow\bar{\mathbb{Q}}_l$
(Section I.2.12 in \cite{Kiehl_Weissauer_Weil_conjectures_perverse_sheaves_and_l'adic_Fourier_transform}).
Hence, we have the following map
\begin{eqnarray*}
\chi^n:\mathcal{D}_{G_{\mathbf{V}}}(E_{\mathbf{V}})&\rightarrow&\underline{\mathcal{F}}^{n}_{\mathbf{V}}\\
\mathcal{L}&\mapsto&\chi^n_{\mathcal{L}}
\end{eqnarray*}
The restriction of this map on the subcategory $\mathcal{Q}_{\mathbf{V}}$ is also denoted by
$$\chi^n:\mathcal{Q}_{\mathbf{V}}\rightarrow\underline{\mathcal{F}}^{n}_{\mathbf{V}}.$$
Lusztig proved that $\chi^n(\mathcal{Q}_{\mathbf{V}})\subset\mathcal{F}^{n}_{\mathbf{V}}$ in \cite{Lusztig_Canonical_bases_and_Hall_algebras}.
Hence, we can define
$\chi^n:\mathcal{Q}_{\mathbf{V}}\rightarrow\mathcal{F}^{n}_{\mathbf{V}}$,
which induces
$\chi:\mathcal{Q}_{\mathbf{V}}\rightarrow\mathcal{F}_{\mathbf{V}}$ naturally.
Hence, we get a map
$\chi_{\mathcal{A}}:K(\mathcal{Q})\rightarrow\mathcal{F}_{\mathcal{A}}$.

Lusztig proved the following proposition.

\begin{proposition}[\cite{Lusztig_Canonical_bases_and_Hall_algebras}]\label{theorem:4.3}
$\chi_{\mathcal{A}}:K(\mathcal{Q})\rightarrow\mathcal{F}_{\mathcal{A}}$
is an isomorphism of $\mathcal{A}$-algebras such that $\chi_{\mathcal{A}}([{\mathcal{L}_i}])=f_i$ and the
following diagram is commutative
$$\xymatrix{K(\mathcal{Q})\ar[r]^-{\lambda_{\mathcal{A}}}\ar[d]^-{\chi_{\mathcal{A}}}&\mathbf{f}_{\mathcal{A}}\ar[ld]^-{\varpi_{\mathcal{A}}}\\
\mathcal{F}_{\mathcal{A}}}$$
\end{proposition}

\subsection{Hall algebra approaches to ${_i\mathbf{f}}$ and ${^i\mathbf{f}}$}\label{subsection:4.2}

Let $i$ be a sink of $Q$. In Section \ref{subsection:3.2}, ${_iE}_\mathbf{V}$ and $j_{\mathbf{V}}:{_iE}_{\mathbf{V}}\rightarrow E_{\mathbf{V}}$ are defined for any $I$-graded $\mathbb{K}$-vector space $\mathbf{V}$.
Similarly, ${_iE}^{F^n}_{\mathbf{V}}$ is defined as the $F_n$-fixed points set in ${_iE}_\mathbf{V}$ and we have
$j_{\mathbf{V}}:{_iE}^{F^n}_{\mathbf{V}}\rightarrow E^{F^n}_{\mathbf{V}}$.

Lusztig also defined ${_i\underline{\mathcal{F}}}^n_{\mathbf{V}}$ as the set of all $G^{F^n}_{\mathbf{V}}$-invariant $\bar{\mathbb{Q}}_l$-functions on ${_iE}^{F^n}_{\mathbf{V}}$.
Similarly to the case in Section \ref{subsection:4.1}, the Hall algebra is denoted by ${_i\underline{\mathcal{F}}}^n=\bigoplus_{\nu\in\mathbb{N}I}{_i\underline{\mathcal{F}}}^n_{\mathbf{V}}$,
the composition subalgebra is denoted by
${_i\mathcal{F}}^n=\bigoplus_{\nu\in\mathbb{N}I}{_i\mathcal{F}}^n_{\mathbf{V}}$
and the generic form is denoted by
${_i\mathcal{F}}:=\bigoplus_{\nu\in\mathbb{N}I}{_i\mathcal{F}}_{\mathbf{V}}$.

Naturally, we have two maps $j^{\ast}_{\mathbf{V}}:\mathcal{F}_\mathbf{V}\rightarrow{_i\mathcal{F}}_\mathbf{V}$ and ${j_{\mathbf{V}}}_{!}:{_i\mathcal{F}}_{\mathbf{V}}\rightarrow \mathcal{F}_{\mathbf{V}}$.
Considering all dimension vectors, we have $j_{!}:{_i\mathcal{F}}\rightarrow\mathcal{F}$ and $j^{\ast}:\mathcal{F}\rightarrow{_i\mathcal{F}}$.

\begin{proposition}[\cite{Lusztig_Canonical_bases_and_Hall_algebras}]\label{theorem:4.5}
We have the following commutative diagram
$$\xymatrix{{_i\mathbf{f}}\ar[r]\ar[d]_{\cong}^-{_i\varpi}&\mathbf{f}\ar[d]_{\cong}^-{\varpi}\ar[r]^-{_i\pi}&{_i\mathbf{f}}\ar[d]_{\cong}^-{_i\varpi}\\
  {_i\mathcal{F}}\ar[r]^-{j_{!}}&\mathcal{F}\ar[r]^-{j^{\ast}}&{_i\mathcal{F}}}$$
where $_i\varpi$ is the isomorphism induced by $\varpi$.
\end{proposition}

Next, we shall prove Proposition \ref{theorem:3.5}.

For any $\mathcal{L}\in\mathcal{D}_{G_{\mathbf{V}}}({_iE}_{\mathbf{V}})$, there is also a function $\chi^n_{\mathcal{L}}:{_iE}^{F^n}_\mathbf{V}\rightarrow\bar{\mathbb{Q}}_l$.
Hence, we have the following map
\begin{eqnarray*}
{_i\chi}^n:\mathcal{D}_{G_{\mathbf{V}}}({_iE}_{\mathbf{V}})&\rightarrow&{_i\underline{\mathcal{F}}}^{n}_{\mathbf{V}}\\
\mathcal{L}&\mapsto&\chi^n_{\mathcal{L}}
\end{eqnarray*}
The restriction of this map on the subcategory ${_i\mathcal{Q}}_{\mathbf{V}}$ is also denoted by
$${_i\chi}^n:{_i\mathcal{Q}}_{\mathbf{V}}\rightarrow{_i\underline{\mathcal{F}}}^{n}_{\mathbf{V}}.$$

\begin{proposition}\label{prop:5.5}
It holds that ${_i\chi}^n({_i\mathcal{Q}}_{\mathbf{V}})\subset{_i\mathcal{F}}^{n}_{\mathbf{V}}$.
\end{proposition}

\begin{proof}[\bf{Proof}]

By the properties of $\chi$ and $_i\chi$ (Theorem III.12.1(5) in \cite{Kiehl_Weissauer_Weil_conjectures_perverse_sheaves_and_l'adic_Fourier_transform}), we have the following commutative diagram
\begin{equation}\label{equation:4.2.2}
\xymatrix{
\mathcal{Q}_{\mathbf{V}}\ar[r]^-{j^{\ast}_\mathbf{V}}\ar[d]^-{\chi^n}&{_i\mathcal{Q}}_{\mathbf{V}}\ar[d]^-{_i\chi^n}\\
\mathcal{F}^n_{\mathbf{V}}\ar[r]^-{j^{\ast}_\mathbf{V}}&{_i\underline{\mathcal{F}}}^{n}_{\mathbf{V}}
}
\end{equation}
By the commutative diagram (\ref{equation:4.2.2}), $j^{\ast}_\mathbf{V}(\mathcal{F}^n_{\mathbf{V}})\subset{_i\mathcal{F}}^{n}_{\mathbf{V}}$ and $j^{\ast}_{\mathbf{V}}(\mathcal{Q}_{\mathbf{V}})={_i\mathcal{Q}}_{\mathbf{V}}$, we have ${_i\chi}^n({_i\mathcal{Q}}_{\mathbf{V}})\subset{_i\mathcal{F}}^{n}_{\mathbf{V}}$.

\end{proof}

Hence, we can define
${_i\chi}^n:{_i\mathcal{Q}}_{\mathbf{V}}\rightarrow{_i\mathcal{F}}^{n}_{\mathbf{V}}$,
which induces
${_i\chi}:{_i\mathcal{Q}}_{\mathbf{V}}\rightarrow{_i\mathcal{F}}_{\mathbf{V}}$
and
${_i\chi_{\mathcal{A}}}:K(_i\mathcal{Q})\rightarrow{_i\mathcal{F}}_{\mathcal{A}}$.

The commutative diagram (\ref{equation:4.2.2}) implies the following proposition.
\begin{proposition}\label{cor:5.2}
We have the following commutative diagram
$$\xymatrix{
K(\mathcal{Q})\ar[r]^-{j^{\ast}}\ar[d]^-{\chi_{\mathcal{A}}}&K(_i\mathcal{Q})\ar[d]^-{{_i\chi}_{\mathcal{A}}}\\
\mathcal{F}_{\mathcal{A}}\ar[r]^-{j^{\ast}}&{_i\mathcal{F}}_{\mathcal{A}}
}$$
\end{proposition}\qed

\begin{proof}[\bf{Proof of Proposition \ref{theorem:3.5}}]
First, we shall prove the following commutative diagram
\begin{displaymath}
\xymatrix{
{_i\mathbf{f}}_{\mathcal{A}}\ar[r]^-{_i\varpi_{\mathcal{A}}}\ar[d]^-{_i\lambda'_{\mathcal{A}}}&{_i\mathcal{F}}_\mathcal{A} \\
K(_i\mathcal{Q})\ar[ru]_-{_i\chi_{\mathcal{A}}}&
}
\end{displaymath}

Consider the following diagram
$$\xymatrix{
\mathbf{f}_{\mathcal{A}}\ar[d]\ar[r]\ar@(dl,ul)[dd] &{_i\mathbf{f}}_{\mathcal{A}}\ar[d]\ar@(dr,ur)[dd] \\
K(\mathcal{Q})\ar[d]\ar[r]&K(_i\mathcal{Q})\ar[d]\\
\mathcal{F}_\mathcal{A}\ar[r]&{_i\mathcal{F}}_\mathcal{A}
}$$
Since three squares and the triangle in the left are commutative, the triangle in the right is also commutative.

Proposition \ref{theorem:4.5} implies that ${_i\varpi_{\mathcal{A}}}:{_i\mathbf{f}}_{\mathcal{A}}\rightarrow {_i\mathcal{F}}_\mathcal{A}$ is isomorphic. Hence ${_i\lambda'_{\mathcal{A}}}:{_i\mathbf{f}}_{\mathcal{A}}\rightarrow K(_i\mathcal{Q})$ is injective. The commutative diagram (\ref{equation:3.2.1}) in the definition of ${_i\lambda'_{\mathcal{A}}}$ implies ${_i\lambda'_{\mathcal{A}}}:{_i\mathbf{f}}_{\mathcal{A}}\rightarrow K(_i\mathcal{Q})$ is surjection. Hence, ${_i\lambda'_{\mathcal{A}}}:{_i\mathbf{f}}_{\mathcal{A}}\rightarrow K(_i\mathcal{Q})$ is isomorphic.

\end{proof}

In the proof, we get the following proposition.

\begin{proposition}
We have the following commutative diagram
$$\xymatrix{K(_i\mathcal{Q})\ar[r]^-{_i\lambda_{\mathcal{A}}}\ar[d]^-{_i\chi_{\mathcal{A}}}&{_i\mathbf{f}}_{\mathcal{A}}\ar[ld]^-{_i\varpi_{\mathcal{A}}}\\
      {_i\mathcal{F}}_{\mathcal{A}}}$$
where all maps are isomorphisms of $\mathcal{A}$-algebras and ${_i\lambda_{\mathcal{A}}}$ is the inverse of ${_i\lambda'_{\mathcal{A}}}$.
\end{proposition}
\qed

Assume that $i$ is a source of $Q$. The notations and results in this case are completely similar to the case that $i$ is a sink. We can define ${^i\underline{\mathcal{F}}}^n_{\mathbf{V}}$,
${^i\mathcal{F}}^n=\bigoplus_{\nu\in\mathbb{N}I}{^i\mathcal{F}}^n_{\mathbf{V}}$
and
${^i\mathcal{F}}=\bigoplus_{\nu\in\mathbb{N}I}{^i\mathcal{F}}_{\mathbf{V}}$.
We also have two maps $j^{\ast}_{\mathbf{V}}:\mathcal{F}_\mathbf{V}\rightarrow{^i\mathcal{F}}_\mathbf{V}$ and ${j_{\mathbf{V}}}_{!}:{^i\mathcal{F}}_{\mathbf{V}}\rightarrow \mathcal{F}_{\mathbf{V}}$.
Considering all dimension vectors, we have $j_{!}:{^i\mathcal{F}}\rightarrow\mathcal{F}$ and $j^{\ast}:\mathcal{F}\rightarrow{^i\mathcal{F}}$.

\begin{proposition}[\cite{Lusztig_Canonical_bases_and_Hall_algebras}]
We have the following commutative diagram
$$\xymatrix{{^i\mathbf{f}}\ar[r]\ar[d]_{\cong}^-{^i\varpi}&\mathbf{f}\ar[d]_{\cong}^-{\varpi}\ar[r]^-{^i\pi}&{^i\mathbf{f}}\ar[d]_{\cong}^-{^i\varpi}\\
  {^i\mathcal{F}}\ar[r]^-{j_{!}}&\mathcal{F}\ar[r]^-{j^{\ast}}&{^i\mathcal{F}}}$$
where $^i\varpi$ is the isomorphism induced by $\varpi$.
\end{proposition}

We can also define
${^i\chi}:{^i\mathcal{Q}}_{\mathbf{V}}\rightarrow{^i\mathcal{F}}_{\mathbf{V}}$
and
${^i\chi_{\mathcal{A}}}:K(^i\mathcal{Q})\rightarrow{^i\mathcal{F}}_{\mathcal{A}}$.

\begin{proposition}
We have the following commutative diagram
$$\xymatrix{
K(\mathcal{Q})\ar[r]^-{j^{\ast}}\ar[d]^-{\chi_{\mathcal{A}}}&K(^i\mathcal{Q})\ar[d]^-{{^i\chi}_{\mathcal{A}}}\\
\mathcal{F}_{\mathcal{A}}\ar[r]^-{j^{\ast}}&{^i\mathcal{F}}_{\mathcal{A}}
}$$
\end{proposition}\qed

\begin{proposition}
We have the following commutative diagram
$$\xymatrix{K(^i\mathcal{Q})\ar[r]^-{^i\lambda_{\mathcal{A}}}\ar[d]^-{^i\chi_{\mathcal{A}}}&{^i\mathbf{f}}_{\mathcal{A}}\ar[ld]^-{^i\varpi_{\mathcal{A}}}\\
      {^i\mathcal{F}}_{\mathcal{A}}}$$
where all maps are isomorphisms of $\mathcal{A}$-algebras and ${^i\lambda_{\mathcal{A}}}$ is the inverse of ${^i\lambda'_{\mathcal{A}}}$.
\end{proposition}\qed

\subsection{Hall algebra approach to $T_i:{_i\mathbf{f}}\rightarrow{^i\mathbf{f}}$ and the proof of Theorem \ref{theorem:3.10}}\label{subsection:4.4}

Let $i$ be a sink of a quiver $Q=(I,H,s,t)$. So $i$ is a source of $Q'=\sigma_iQ=(I,H',s,t)$. For any $\nu$ and $\nu'\in\mathbb{N}I$ such that  $\nu'=s_i\nu$,
and two  $I$-graded $\mathbb{K}$-vector spaces $\mathbf{V}$ and $\mathbf{V}'$ such that $\underline{\dim}\mathbf{V}=\nu$ and $\underline{\dim}\mathbf{V}'=\nu'$,
the following correspondence is considered  in Section \ref{subsection:3.4}
$$\xymatrix{
_iE_{\mathbf{V},Q}&Z_{\mathbf{V}\mathbf{V}'}\ar[l]_-{\alpha}\ar[r]^-{\beta}&^iE_{\mathbf{V}',Q'}}.
$$
Similarly, $Z^{F^n}_{\mathbf{V}\mathbf{V}'}$ is defined as the $F_n$-fixed points set in $Z_{\mathbf{V}\mathbf{V}'}$ and we have
$$
\xymatrix{
_iE^{F^n}_{\mathbf{V},Q}&Z^{F^n}_{\mathbf{V}\mathbf{V}'}\ar[l]_-{\alpha}\ar[r]^-{\beta}&^iE^{F^n}_{\mathbf{V}',Q'}}.
$$
Note that $\alpha$ and $\beta$ are principal bundles with fibers $Aut(V'_i)$ and  $Aut(V_i)$ respectively. Hence, for any $f\in{_i\underline{\mathcal{F}}}^n_\mathbf{V}$, there exists a unique $g\in{^i\underline{\mathcal{F}}}^n_{\mathbf{V}'}$ such that $\alpha^{\ast}(f)=\beta^{\ast}(g)$. Define
\begin{eqnarray*}
\omega_i:{_i\underline{\mathcal{F}}}^n_\mathbf{V}&\rightarrow&{^i\underline{\mathcal{F}}}^n_{\mathbf{V}'}\\
f&\mapsto&(p^n)^{-\frac{s(\mathbf{V})}{2}}g
\end{eqnarray*}
Lusztig proved that $\omega_i({_i\mathcal{F}}^n_\mathbf{V})\subset{^i\mathcal{F}}^n_{\mathbf{V}'}$. Hence, we have
$\omega_i:{_i\mathcal{F}}^n_\mathbf{V}\rightarrow{^i\mathcal{F}}^n_{\mathbf{V}'}$ and
$\omega_i:{_i\mathcal{F}}_\mathbf{V}\rightarrow{^i\mathcal{F}}_{\mathbf{V}'}$. Considering all dimension vectors, we have
$\omega_i:{_i\mathcal{F}}\rightarrow{^i\mathcal{F}}$.

Lusztig proved the following theorem.

\begin{theorem}[\cite{Lusztig_Canonical_bases_and_Hall_algebras}]\label{theorem:4.17}
We have the following commutative diagram
$$\xymatrix{{_i\mathbf{f}}\ar[r]^-{T_i}\ar[d]^-{_i\varpi}&{^i\mathbf{f}}\ar[d]^-{^i\varpi}\\_i\mathcal{F}\ar[r]^-{\omega_i}&^i\mathcal{F}}$$
\end{theorem}

\begin{proof}[\bf{Proof of Proposition \ref{proposition:3.9}}]

By the properties of ${_i\chi^n}$ and ${^i\chi^n}$ (Theorem III.12.1(4,5) in \cite{Kiehl_Weissauer_Weil_conjectures_perverse_sheaves_and_l'adic_Fourier_transform}), we have the following commutative diagram
$$\xymatrix{\mathcal{D}_{G_{\mathbf{V}}}(_iE_{\mathbf{V}})\ar[r]^-{\tilde{\omega}_i}\ar[d]^-{_i\chi^n}
&\mathcal{D}_{G_{\mathbf{V}'}}(^iE_{\mathbf{V}'})\ar[d]^-{^i\chi^n}\\
_i\underline{\mathcal{F}}^n_\mathbf{V}\ar[r]^-{\omega^n_i}
&^i\underline{\mathcal{F}}^n_{\mathbf{V}'}}$$
Hence, we have
$$\xymatrix@=35pt{\mathcal{D}_{G_{\mathbf{V}}}(_iE_{\mathbf{V}})\ar[r]^-{\tilde{\omega}_i}\ar[d]^-{\prod_{n\in\mathbb{Z}_{\geq1}}{_i\chi}^n}
&\mathcal{D}_{G_{\mathbf{V}'}}(^iE_{\mathbf{V}'})\ar[d]^-{\prod_{n\in\mathbb{Z}_{\geq1}}{^i\chi}^n}\\
\prod_{n\in\mathbb{Z}_{\geq1}}{}_i\underline{\mathcal{F}}^n_\mathbf{V}\ar[r]^-{\prod_{n\in\mathbb{Z}_{\geq1}}\omega^n_i}
&\prod_{n\in\mathbb{Z}_{\geq1}}^i\underline{\mathcal{F}}^n_{\mathbf{V}'}}$$
By Proposition \ref{prop:5.5}, ${_i\chi}^n({_i\mathcal{Q}}_{\mathbf{V}})\subset{_i\mathcal{F}}^{n}_{\mathbf{V}}$. Hence, we have
$$\xymatrix{{_i\mathcal{Q}}_\mathbf{V}\ar[r]^-{\tilde{\omega}_i}\ar[d]^-{_i\chi}
&\mathcal{D}_{G_{\mathbf{V}'}}(^iE_{\mathbf{V}'})\ar[d]^-{\prod_{n\in\mathbb{Z}_{\geq1}}{^i\chi}^n}\\
{_i\mathcal{F}}_\mathbf{V}\ar[r]^-{\omega_i}
&\prod_{n\in\mathbb{Z}_{\geq1}}{^i\underline{\mathcal{F}}}^n_{\mathbf{V}'}}$$
Hence,
$$
(\prod_{n\in\mathbb{Z}_{\geq1}}{^i\chi}^n)\circ\tilde{\omega}_i({_i\mathcal{Q}}_\mathbf{V})\subset\omega_i\circ{_i\chi}({_i\mathcal{Q}}_\mathbf{V}).
$$
Since
$
\omega_i({_i\mathcal{F}}_\mathbf{V})\subset{^i\mathcal{F}}_{\mathbf{V}'},
$
$$
(\prod_{n\in\mathbb{Z}_{\geq1}}{^i\chi}^n)\circ\tilde{\omega}_i({_i\mathcal{Q}}_\mathbf{V})\subset{^i\mathcal{F}}_{\mathbf{V}'}.
$$

For any two semisimple perverse sheaves $\mathcal{L}$ and $\mathcal{L}'$ in $\mathcal{D}_{G_{\mathbf{V}'}}(^iE_{\mathbf{V}'})$ such that
$$(\prod_{n\in\mathbb{Z}_{\geq1}}{^i\chi}^n)(\mathcal{L})=(\prod_{n\in\mathbb{Z}_{\geq1}}{^i\chi}^n)(\mathcal{L}'),$$
$\mathcal{L}$ is isomorphic to $\mathcal{L}'$ by Theorem III.12.1(3) in \cite{Kiehl_Weissauer_Weil_conjectures_perverse_sheaves_and_l'adic_Fourier_transform}.
Since $(\prod_{n\in\mathbb{Z}_{\geq1}}{^i\chi}^n)({^i\mathcal{Q}}_{\mathbf{V}'})={^i\mathcal{F}}_{\mathbf{V}'}$ and the objects in $\tilde{\omega}_i({_i\mathcal{Q}}_\mathbf{V})$ are semisimple,
$
\tilde{\omega}_i({_i\mathcal{Q}}_\mathbf{V})\subset{^i\mathcal{Q}}_{\mathbf{V}'}.
$

\end{proof}

\begin{proposition}\label{theorem:4.18}
We have the following commutative diagram
$$\xymatrix{K({_i\mathcal{Q}})\ar[r]^-{\tilde{\omega}_i}\ar[d]^-{_i\chi}
&K({^i\mathcal{Q}})\ar[d]^-{{^i\chi}}\\
{_i\mathcal{F}}_{\mathcal{A}}\ar[r]^-{\omega_i}
&{^i\mathcal{F}}_{\mathcal{A}}}$$
\end{proposition}

\begin{proof}[\bf{Proof}]

By the properties of ${_i\chi^n}$ and ${^i\chi^n}$, we have the following commutative diagram
$$\xymatrix{{_i\mathcal{Q}}_{\mathbf{V}}\ar[r]^-{\tilde{\omega}_i}\ar[d]^-{_i\chi^n}
&{^i\mathcal{Q}}_{\mathbf{V}'}\ar[d]^-{{^i\chi}^n}\\
{_i\mathcal{F}}^n_\mathbf{V}\ar[r]^-{\omega^n_i}
&{^i\mathcal{F}}^n_{\mathbf{V}'}}$$
Hence, we get the commutative diagram in this proposition.

\end{proof}

At last, Theorem \ref{theorem:4.17} and Proposition \ref{theorem:4.18} imply Theorem \ref{theorem:3.10}.

\section{Projective resolutions of a kind of standard modules}

\subsection{KLR algebras}\label{subsection:5.1}

First let us review the definitions of KLR algebras (\cite{Khovanov_Lauda_A_diagrammatic_approach_to_categorification_of_quantum_groups_I,Varagnolo_Vasserot_Canonical_bases_and_KLR-algebras}).

Let $Q=(I,H,s,t)$ be a quiver corresponding to the Lusztig's algebra $\mathbf{f}$.
Let $\mathbb{K}$ be an algebraic closed field. Fix an $I$-graded $\mathbb{K}$-vector space $\mathbf{V}$ such that $\underline{\dim}\mathbf{V}=\nu\in\mathbb{N}I$.
In Section \ref{subsection:3.1}, the semisimple perverse sheaves $\mathcal{L}_{\mathbf{i}}\in\mathcal{D}_{G_{\mathbf{V}}}(E_{\mathbf{V}})$ are defined for all $\mathbf{i}\in I^{\nu}$.
Let $$\mathcal{L}_{\nu}=\bigoplus_{\mathbf{i}\in I^{\nu}}\mathcal{L}_{\mathbf{i}}.$$
The KLR algebra $\mathbf{R}_{\nu}$ is defined as
$$\mathbf{R}_{\nu}=\bigoplus_{k\in\mathbb{Z}}\textrm{Ext}^{k}_{G_{\mathbf{V}}}(\mathcal{L}_{\nu},\mathcal{L}_{\nu}).$$
$\mathbf{R}_{\nu}$ is a graded algebra and the degree of any element in $\textrm{Ext}^{k}_{G_{\mathbf{V}}}(\mathcal{L}_{\nu},\mathcal{L}_{\nu})$ is $k$.

Let $\mathbf{R}_{\nu}\textrm{-gmod}$ be the category of graded $\mathbf{R}_{\nu}$-modules and $\mathbf{R}_{\nu}\textrm{-proj}$ be the category of finitely generated graded projective $\mathbf{R}_{\nu}$-modules. Let $K(\mathbf{R}_{\nu}\textrm{-proj})$ be the Grothendieck group of $\mathbf{R}_{\nu}\textrm{-proj}$.

Define $v^{\pm}[P]=[P[\pm1]]$. So $K(\mathbf{R}_{\nu}\textrm{-proj})$ is a free $\mathcal{A}$-module. Define $$K(\mathbf{R}\textrm{-proj})=\bigoplus_{\nu\in\mathbb{N}I}K(\mathbf{R}_{\nu}\textrm{-proj}).$$

For $\nu,\nu',\nu''\in\mathbb{N}I$ such that $\nu=\nu'+\nu''$ and three $I$-graded $\mathbb{K}$-vector spaces $\mathbf{V}$, $\mathbf{V}'$, $\mathbf{V}''$ such that $\underline{\dim}\mathbf{V}=\nu$, $\underline{\dim}\mathbf{V}'=\nu'$, $\underline{\dim}\mathbf{V}''=\nu''$, Khovanov and Lauda (\cite{Khovanov_Lauda_A_diagrammatic_approach_to_categorification_of_quantum_groups_I}) defined a functor
$$\textrm{Ind}_{\nu',\nu''}:\mathbf{R}_{\nu'}\textrm{-proj}\times\mathbf{R}_{\nu''}\textrm{-proj}\rightarrow\mathbf{R}_{\nu}\textrm{-proj},$$
which induces an $\mathcal{A}$-bilinear multiplication
$$[\textrm{Ind}_{\nu',\nu''}]:K(\mathbf{R}_{\nu'}\textrm{-proj})\otimes_{\mathcal{A}}K(\mathbf{R}_{\nu''}\textrm{-proj})\rightarrow K(\mathbf{R}_{\nu}\textrm{-proj}).$$
Khovanov and Lauda (\cite{Khovanov_Lauda_A_diagrammatic_approach_to_categorification_of_quantum_groups_I}) proved that
$K(\mathbf{R}\textrm{-proj})$ becomes an associative $\mathcal{A}$-algebra.

For any $\mathbf{y}\in Y_{\nu}$, let $$P_{\mathbf{y}}=\bigoplus_{k\in\mathbb{Z}}\textrm{Ext}^{k}_{G_{\mathbf{V}}}(\mathcal{L}_{\mathbf{y}},\mathcal{L}_{\nu}).$$

\begin{theorem}[\cite{Khovanov_Lauda_A_diagrammatic_approach_to_categorification_of_quantum_groups_I,Rouquier_2-Kac-Moody_algebras}]
There is a unique isomorphism of  $\mathcal{A}$-algebras
$$\gamma_{\mathcal{A}}:\mathbf{f}_{\mathcal{A}}\rightarrow K(\mathbf{R}\textrm{-proj})$$
such that $\gamma_{\mathcal{A}}(\theta_{\mathbf{y}})=P_{\mathbf{y}}$ for all $\mathbf{y}\in Y_{\nu}$.
\end{theorem}

Let $\mathbf{B}_{\mathbb{Z}}=\{v^sb\,\,|\,\,b\in\mathbf{B}, s\in\mathbb{Z}\}$, which is a $\mathbb{Z}$-basis of $\mathbf{f}_{\mathcal{A}}$.
Varagnolo, Vasserot and Rouquier proved the following theorem.

\begin{theorem}[\cite{Varagnolo_Vasserot_Canonical_bases_and_KLR-algebras,Rouquier_Quiver_Hecke_algebras_and_2-Lie_algebras}]
The map $\gamma_{\mathcal{A}}$ takes $\mathbf{B}_{\mathbb{Z}}$ to the $\mathbb{Z}$-basis of $K(\mathbf{R}\textrm{-proj})$ consisting of all indecomposable projective modules.
\end{theorem}

\subsection{Projective resolutions}\label{subsection:5.2}

Let $i$ and $j$ be two vertices of the quiver $Q$ such that there are no arrows from $i$ to $j$.
Let $N=\#\{j\rightarrow i\}$ and $m$ be a non-negative integer such that $m\leq N$.
Let $\nu^{(m)}=mi+j\in\mathbb{N}I$. Fix an $I$-graded $\mathbb{K}$-vector space $\mathbf{V}^{(m)}$ such that $\underline{\dim}\mathbf{V}^{(m)}=\nu^{(m)}$.

Denote by $\mathbf{1}_{_iE_{\mathbf{V}^{(m)}}}\in\mathcal{D}_{G_{\mathbf{V}^{(m)}}}(_iE_{\mathbf{V}^{(m)}})$ the constant sheaf on $_iE_{\mathbf{V}^{(m)}}$.
The following functor is defined in Section \ref{subsection:3.2}:
$$j_{\mathbf{V}^{(m)}!}:\mathcal{D}_{G_{\mathbf{V}^{(m)}}}(_iE_{\mathbf{V}^{(m)}})\rightarrow \mathcal{D}_{G_{\mathbf{V}^{(m)}}}(E_{\mathbf{V}^{(m)}}).$$
Define
$$\mathcal{E}^{(m)}=j_{\mathbf{V}^{(m)}!}(v^{-mN}\mathbf{1}_{_iE_{\mathbf{V}^{(m)}}})\in\mathcal{D}_{G_{\mathbf{V}^{(m)}}}(E_{\mathbf{V}^{(m)}})$$
and
$$K_m=\bigoplus_{k\in\mathbb{Z}}\textrm{Ext}^{k}_{G_{\mathbf{V}^{(m)}}}(\mathcal{E}^{(m)},\mathcal{L}_{\nu^{(m)}}).$$
$K_m$ is an object in $\mathbf{R}_{\nu^{(m)}}\textrm{-gmod}$ for any $m$. Note that $K_m$ is a standard module in the sense of Kato (\cite{Kato_PBW_bases_and_KLR_algebras}). We shall give projective resolutions of these standard modules.
For convenience, the complex $j_{\mathbf{V}^{(m)}!}(\mathbf{1}_{_iE_{\mathbf{V}^{(m)}}})\in\mathcal{D}_{G_{\mathbf{V}^{(m)}}}(E_{\mathbf{V}^{(m)}})$ is also denoted by
$\mathbf{1}_{_iE_{\mathbf{V}^{(m)}}}$.

For each $m\geq p\in\mathbb{N}$, consider the following variety
$$\tilde{S}^{(m)}_p=\{(x,W)\,\,|\,\,x\in E_{\mathbf{V}^{(m)}},\,W\subset V_i,\,\dim(W)=p,\,\textrm{Im}\bigoplus_{h\in H,t(h)=i}x_{h}\subset W\}.$$
Let $\pi_p:\tilde{S}^{(m)}_p\rightarrow E_{\mathbf{V}^{(m)}}$ be the projection taking $(x,W)$ to $x$ and $S^{(m)}_p=\textrm{Im}\pi_p$.

By the definitions of $S^{(m)}_p$, we have
$${E_{\mathbf{V}^{(m)}}}=S^{(m)}_m\supset{S^{(m)}_{m-1}}\supset{S^{(m)}_{m-2}}\supset\cdots\supset{S^{(m)}_0}.$$
For each $1\leq p\leq m$, let $$\mathcal{N}^{(m)}_{p}={S^{(m)}_{p}}\backslash{S^{(m)}_{p-1}}.$$
Denote by $i^{(m)}_p:{S^{(m)}_{p-1}}\rightarrow{S^{(m)}_{p}}$ the close embedding and $j^{(m)}_p:\mathcal{N}^{(m)}_{p}\rightarrow{S^{(m)}_{p}}$ the open embedding.

Define
$$I^{(m)}_p={(\pi_p)}_!(\mathbf{1}_{\tilde{S}^{(m)}_p})[\dim\tilde{S}^{(m)}_p].$$
In \cite{Lusztig_Quivers_perverse_sheaves_and_the_quantized_enveloping_algebras}, Lusztig proved that $I^{(m)}_p$ are semisimple perverse sheaves in $\mathcal{D}_{G_{\mathbf{V}^{(m)}}}(E_{\mathbf{V}^{(m)}})$. Hence $I^{(m)}_p$ correspond to projective modules in $\mathbf{R}_{\nu}\textrm{-proj}$.

The following theorem is the main result in this section.

\begin{theorem}\label{theorem:5.4}
For $\mathcal{E}^{(m)}$, there exists $s_{m}\in\mathbb{N}$. For each $s_m\geq p\in\mathbb{N}$, there exists $\mathcal{E}_p^{(m)}\in\mathcal{D}_{G_{\mathbf{V}^{(m)}}}(E_{\mathbf{V}^{(m)}})$ such that
\begin{enumerate}
\item[(1)]$\mathcal{E}_{s_m}^{(m)}=\mathcal{E}^{(m)}$ and $\mathcal{E}_{0}^{(m)}$ is the direct sum of some semisimple perverse sheaves of the form $I^{(m)}_{p'}[l]$;
\item[(2)]for each $p\geq1$, there exists a  distinguished triangle
   \begin{displaymath}
       \xymatrix{
      \mathcal{E}^{(m)}_{p}\ar[r]&{\mathcal{G}^{(m)}_{p}}\ar[r]&\mathcal{E}^{(m)}_{p-1}\ar[r]&,
      }
      \end{displaymath}
      where $\mathcal{G}^{(m)}_{p}$ is the direct sum of some semisimple perverse sheaves of the form $I^{(m)}_{p'}[l]$.
\end{enumerate}
\end{theorem}

The proof of Theorem \ref{theorem:5.4} will be given in Section \ref{subsection:5.3}.

Let
$${P^{(m)}_{0}}=\bigoplus_{k\in\mathbb{Z}}\textrm{Ext}^{k}_{G_{\mathbf{V}^{(m)}}}(\mathcal{E}_{0}^{(m)},\mathcal{L}_{\nu^{(m)}})$$
and
$${P^{(m)}_{s}}=\bigoplus_{k\in\mathbb{Z}}\textrm{Ext}^{k}_{G_{\mathbf{V}^{(m)}}}(\mathcal{G}^{(m)}_{p},\mathcal{L}_{\nu^{(m)}})\,\,\,\,\,(1\leq s\leq m),$$
which are projective modules in $\mathbf{R}_{\nu^{(m)}}\textrm{-proj}$.

As a corollary of Theorem \ref{theorem:5.4}, we have the following theorem.

\begin{theorem}\label{theorem:5.5}
For any $N\geq m\in\mathbb{N}$, there exists a finite length projective resolution of $K_m$:
\begin{displaymath}
\xymatrix{
0\ar[r]&{P^{(m)}_{0}}\ar[r]&{P^{(m)}_{1}}\ar[r]&\cdots\ar[r]&{P^{(m)}_{s_m-1}}\ar[r]&{P^{(m)}_{s_m}}\ar[r]&K_m\ar[r]&0.
}
\end{displaymath}
\end{theorem}\qed

In the case of finite type, Kato proved that the projective dimension of any standard module is finite (\cite{Kato_An_algebraic_study_of_extension_algebra,Kato_PBW_bases_and_KLR_algebras}).
Theorem \ref{theorem:5.5} show that the projective dimensions of a kind of standard modules are also finite in the general case.

\subsection{The proof of Theorem \ref{theorem:5.4}}\label{subsection:5.3}

For convenience, a sheaf $\mathcal{A}\in\mathcal{D}_{G_{\mathbf{V}^{(m)}}}(E_{\mathbf{V}^{(m)}})$ is called with Property \textbf{A}(m), if $\mathcal{A}$ satisfies the following conditions.
There exists $s_{\mathcal{A}}\in\mathbb{N}$.
For each $s_{\mathcal{A}}\geq p\in\mathbb{N}$, there exists $\mathcal{A}_p\in\mathcal{D}_{G_{\mathbf{V}^{(m)}}}(E_{\mathbf{V}^{(m)}})$ such that
\begin{enumerate}
\item[(1)]$\mathcal{A}_{s_{\mathcal{A}}}=\mathcal{A}$ and $\mathcal{A}_{0}$ is the direct sum of some semisimple perverse sheaves of the form $I^{(m)}_{p'}[l]$;
\item[(2)]for each $p\geq1$, there exists a distinguished triangle
   \begin{displaymath}
       \xymatrix{
      \mathcal{A}_{p}\ar[r]&{\mathcal{G}^{\mathcal{A}}_{p}}\ar[r]&\mathcal{A}_{p-1}\ar[r]&,
      }
      \end{displaymath}
      where $\mathcal{G}^{\mathcal{A}}_{p}$ is the direct sum of some semisimple perverse sheaves of the form $I^{(m)}_{p'}[l]$.
\end{enumerate}

Theorem \ref{theorem:5.4} means that $\mathcal{E}^{(m)}$ is with Property \textbf{A}(m).

For the proof of Theorem \ref{theorem:5.4}, we need the following lemma.

\begin{lemma}\label{lemma:5.2}
Fix any distinguished triangle
\begin{displaymath}
\xymatrix{
\mathcal{A}\ar[r]&\mathcal{A}'\ar[r]&\mathcal{A}''\ar[r]&,
}
\end{displaymath}
where $\mathcal{A}, \mathcal{A}', \mathcal{A}''\in\mathcal{D}_{G_{\mathbf{V}^{(m)}}}(E_{\mathbf{V}^{(m)}})$.
If $\mathcal{A}$ and $\mathcal{A}''$ are with Property \textbf{A}(m), $\mathcal{A}'$ is with Property \textbf{A}(m) and $s_{\mathcal{A}'}=s_{\mathcal{A}}+s_{\mathcal{A}''}+1$.
\end{lemma}

\begin{proof}[\bf{Proof}]
We shall prove this lemma by induction on $s_{\mathcal{A}''}$.

\item[\textbf{(1)}] For $s_{\mathcal{A}''}=0$, $\mathcal{A}''$ is the direct sum of some semisimple perverse sheaves of the form $I^{(m)}_{p'}[l]$.
    Let $\mathcal{A}'_{s_{\mathcal{A}'}}=\mathcal{A}'$ and $\mathcal{A}'_{p}=\mathcal{A}_{p}[1]$
    for any $0\leq p\leq s_{\mathcal{A}}=s_{\mathcal{A}'}-1$.
    Let ${\mathcal{G}^{\mathcal{A}'}_{s_{\mathcal{A}'}}}=\mathcal{A}''$ and
    ${\mathcal{G}^{\mathcal{A}'}_{p}}={\mathcal{G}^{\mathcal{A}}_{p}}[1]$ for any $1\leq p\leq s_{\mathcal{A}}=s_{\mathcal{A}'}-1$. The distinguished triangle
    \begin{displaymath}
    \xymatrix{
    \mathcal{A}\ar[r]&\mathcal{A}'\ar[r]&\mathcal{A}''\ar[r]&
    }
    \end{displaymath}
    implies
    \begin{displaymath}
    \xymatrix{
    \mathcal{A}'_{s_{\mathcal{A}'}}\ar[r]&{\mathcal{G}^{\mathcal{A}'}_{s_{\mathcal{A}'}}}
    \ar[r]&\mathcal{A}'_{s_{\mathcal{A}'}-1}\ar[r]&
    }
    \end{displaymath}
    and the distinguished triangles
    \begin{displaymath}
    \xymatrix{
    \mathcal{A}_{p}\ar[r]&{\mathcal{G}^{\mathcal{A}}_{p}}\ar[r]&\mathcal{A}_{p-1}\ar[r]&
    }
    \end{displaymath}
    imply
    \begin{displaymath}
    \xymatrix{
    \mathcal{A}'_{p}\ar[r]&{\mathcal{G}^{\mathcal{A}}_{p}}\ar[r]&\mathcal{A}'_{p-1}\ar[r]&
    }
    \end{displaymath}
    for $1\leq p\leq s_{\mathcal{A}'}-1$. Hence, $\mathcal{A}'$ is with Property \textbf{A}(m).
\item[\textbf{(2)}] Assume that the lemma is true for $s_{\mathcal{A}''}<k$, we shall prove the lemma for $s_{\mathcal{A}''}=k$.

    Now we have the following two distinguished triangles
    \begin{displaymath}
    \xymatrix{
    \mathcal{A}'\ar[r]^{u}&\mathcal{A}''\ar[r]&\mathcal{A}[1]\ar[r]&
    }
    \end{displaymath}
    and
    \begin{displaymath}
    \xymatrix{
    \mathcal{A}''\ar[r]^{v}&\mathcal{G}^{\mathcal{A}''}_{k}\ar[r]&\mathcal{A}''_{k-1}\ar[r]&.
    }
    \end{displaymath}
    Then we can construct the following distinguished triangle
    \begin{displaymath}
    \xymatrix{
    \mathcal{A}'\ar[r]^{vu}&\mathcal{G}^{\mathcal{A}''}_{k}\ar[r]&\mathcal{B}\ar[r]&.
    }
    \end{displaymath}
    By the octahedral axiom, there exist two maps $f:\mathcal{A}[1]\rightarrow\mathcal{B}$ and $g:\mathcal{B}\rightarrow\mathcal{A}''_{k-1}$ such that the following diagram commutes and the third row is a distinguished triangle
    \begin{displaymath}
    \xymatrix{
    \mathcal{A}'\ar[r]^{\textrm{id}}\ar[d]^{u}&\mathcal{A}'\ar[d]^{vu}&&\\
    \mathcal{A}''\ar[r]^{v}\ar[d]&\mathcal{G}^{\mathcal{A}''}_k\ar[r]\ar[d]&\mathcal{A}''_{k-1}\ar[r]\ar[d]^{\textrm{id}}&\\
    \mathcal{A}[1]\ar[r]^{f}\ar[d]&\mathcal{B}\ar[r]^{g}\ar[d]&\mathcal{A}''_{k-1}\ar[r]&\\
    &&&
    }
    \end{displaymath}
    Consider the following distinguished triangle
    \begin{displaymath}
    \xymatrix{
    \mathcal{A}[1]\ar[r]^{f}&\mathcal{B}\ar[r]^{g}&\mathcal{A}''_{k-1}\ar[r]&.
    }
    \end{displaymath}
    Since $\mathcal{A}''_{k}=\mathcal{A}''$ is with Property \textbf{A}(m), $\mathcal{A}''_{k-1}$ is also with Property \textbf{A}(m) and $s_{\mathcal{A}''_{k-1}}=k-1$.
    By the induction hypothesis, $\mathcal{B}$ is with Property \textbf{A}(m) and $s_{\mathcal{B}}=s_{\mathcal{A}}+k$.
    Hence, for each $s_{\mathcal{B}}\geq p\in\mathbb{N}$, there exists $\mathcal{B}_p\in\mathcal{D}_{G_{\mathbf{V}^{(m)}}}(E_{\mathbf{V}^{(m)}})$ such that
    \begin{enumerate}
    \item[1)]$\mathcal{B}_{s_{\mathcal{B}}}=\mathcal{B}$ and $\mathcal{B}_{0}$ is the direct sum of some semisimple perverse sheaves of the form $I^{(m)}_{p'}[l]$;
    \item[2)]for each $p\geq1$, there exists a distinguished triangle
       \begin{displaymath}
       \xymatrix{
       \mathcal{B}_{p}\ar[r]&{\mathcal{G}^{\mathcal{B}}_{p}}\ar[r]&\mathcal{B}_{p-1}\ar[r]&,
       }
       \end{displaymath}
       where $\mathcal{G}^{\mathcal{B}}_{p}$ is the direct sum of some semisimple perverse sheaves of the form $I^{(m)}_{p'}[l]$.
       \end{enumerate}
       Note that $s_{\mathcal{A}'}=s_{\mathcal{B}}+1$.
       Let $\mathcal{A}'_{s_{\mathcal{A}'}}=\mathcal{A}'$ and $\mathcal{A}'_{p}=\mathcal{B}_{p}$
       for any $0\leq p\leq s_{\mathcal{B}}=s_{\mathcal{A}'}-1$.
       Let ${\mathcal{G}^{\mathcal{A}'}_{s_{\mathcal{A}'}}}=\mathcal{G}^{\mathcal{A}''}_{k}$ and $\mathcal{G}^{\mathcal{A}'}_{p}=\mathcal{G}^{\mathcal{B}}_{p}$ for any $1\leq p\leq s_{\mathcal{B}}=s_{\mathcal{A}'}-1$.
       The distinguished triangle
       \begin{displaymath}
       \xymatrix{
       \mathcal{A}'\ar[r]^{vu}&\mathcal{G}^{\mathcal{A}''}_{k}\ar[r]&\mathcal{B}\ar[r]&
       }
       \end{displaymath}
       implies
       \begin{displaymath}
       \xymatrix{
       \mathcal{A}'_{s_{\mathcal{A}'}}\ar[r]^{vu}&{\mathcal{G}^{\mathcal{A}'}_{s_{\mathcal{A}'}}}\ar[r]&\mathcal{A}'_{s_{\mathcal{A}'}-1}\ar[r]&
       }
       \end{displaymath}
       and the distinguished triangles
       \begin{displaymath}
       \xymatrix{
       \mathcal{B}_{p}\ar[r]&{\mathcal{G}^{\mathcal{B}}_{p}}\ar[r]&\mathcal{B}_{p-1}\ar[r]&,
       }
       \end{displaymath}
       imply
    \begin{displaymath}
    \xymatrix{
    \mathcal{A}'_{p}\ar[r]&{\mathcal{G}^{\mathcal{A}}_{p}}\ar[r]&\mathcal{A}'_{p-1}\ar[r]&
    }
    \end{displaymath}
    for $1\leq p\leq s_{\mathcal{A}'}-1$. Hence, $\mathcal{A}'$ is with Property \textbf{A}{(m)}.

By induction, the proof is finished.

\end{proof}

For the proof of Theorem \ref{theorem:5.4}, we also need the following proposition.

\begin{proposition}\label{theorem:5.3}
For each $m\geq p\in\mathbb{N}$, there exists $\mathcal{C}_p^{(m)}\in\mathcal{D}_{G_{\mathbf{V}^{(m)}}}(E_{\mathbf{V}^{(m)}})$ such that
\begin{enumerate}
\item[(1)]$\mathcal{C}_m^{(m)}=I^{(m)}_m$ and $\mathcal{C}^{(m)}_{0}=v^{-mN}({\mathcal{L}}_{mi}\circledast\mathbf{1}_{_iE_{\mathbf{V}^{(0)}}})$;
\item[(2)]for each $p\geq1$, there exists a  distinguished triangle
    \begin{displaymath}
       \xymatrix{
      v^{a^{(m)}_p}({\mathcal{L}}_{(m-p)i}\circledast\mathbf{1}_{_iE_{\mathbf{V}^{(p)}}})\ar[r]
      &{\mathcal{C}_p^{(m)}}\ar[r]
      &\mathcal{C}_{p-1}^{(m)}\ar[r]
      &,
      }
  \end{displaymath}
    where $a^{(m)}_p=p(m-p)-mN$.
\end{enumerate}
\end{proposition}

\begin{proof}[\bf{Proof}]
We shall construct $\mathcal{C}_p^{(m)}$ for each $p$ by induction.

\item[\textbf{(1)}] For $p=m$, let $\mathcal{C}_m^{(m)}=I^{(m)}_m$. It is clear that $I^{(m)}_m\simeq v^{-mN}\mathbf{1}_{E_{\mathbf{V}^{(m)}}}$, that is $\mathcal{C}_m^{(m)}\simeq v^{a^{(m)}_m}\mathbf{1}_{E_{\mathbf{V}^{(m)}}}$.
\item[\textbf{(2)}] For each $p<m$, we shall construct $\mathcal{C}_p^{(m)}$ and show that it satisfies the following conditions:
    \begin{enumerate}
   \item[1)] there exists a distinguished triangle
    \begin{displaymath}
       \xymatrix{
      v^{a^{(m)}_{p+1}}({\mathcal{L}}_{(m-p-1)i}\circledast\mathbf{1}_{_iE_{\mathbf{V}^{(p+1)}}})\ar[r]
      &{\mathcal{C}_{p+1}^{(m)}}\ar[r]
      &\mathcal{C}_{p}^{(m)}\ar[r]
      &;
      }
  \end{displaymath}
   \item[2)] $\mathcal{C}_{p}^{(m)}={\mathcal{L}}_{(m-p)i}\circledast\hat{\mathcal{C}}_{p}^{(m)}$,
       where $\hat{\mathcal{C}}_{p}^{(m)}\in\mathcal{D}_{G_{\mathbf{V}^{(p)}}}(E_{\mathbf{V}^{(p)}})$ and
       $\hat{\mathcal{C}}_{p}^{(m)}\simeq v^{a^{(m)}_p}\mathbf{1}_{E_{\mathbf{V}^{(p)}}}$.
   \end{enumerate}
   First, We construct $\mathcal{C}_{p}^{(m)}$ for $p=m-1$. There is a distinguished triangle
      \begin{equation}\label{equation:5.2.1}
       \xymatrix{
      {(j^{(m)}_m)_!(j^{(m)}_m)^!(\mathcal{C}_m^{(m)})}\ar[r]
      &\mathcal{C}_m^{(m)}\ar[r]
      &{(i^{(m)}_m)_\ast(i^{(m)}_m)^\ast(\mathcal{C}_m^{(m)})}\ar[r]
      &.
      }
      \end{equation}
   Since $\mathcal{C}_m^{(m)}\simeq v^{a^{(m)}_m}\mathbf{1}_{E_{\mathbf{V}^{(m)}}}$,
      $${(j^{(m)}_m)_!(j^{(m)}_m)^!(\mathcal{C}_m^{(m)})}\simeq v^{a^{(m)}_m}\mathbf{1}_{_iE_{\mathbf{V}^{(m)}}},$$
   and
      $$(i^{(m)}_m)_\ast(i^{(m)}_m)^\ast(\mathcal{C}_m^{(m)})\simeq v^{a^{(m)}_m}\mathbf{1}_{S^{(m)}_{m-1}}.$$
   Let $\mathcal{C}_{m-1}^{(m)}=(i^{(m)}_m)_\ast(i^{(m)}_m)^\ast(\mathcal{C}_m^{(m)})$.
   By (\ref{equation:5.2.1}), there exists a distinguished triangle
   \begin{displaymath}
       \xymatrix{
      v^{a^{(m)}_m}\mathbf{1}_{_iE_{\mathbf{V}^{(m)}}}\ar[r]
      &\mathcal{C}_m^{(m)}\ar[r]
      &\mathcal{C}_{m-1}^{(m)}\ar[r]
      &.
      }
   \end{displaymath}
   Since the support of $\mathcal{C}_{m-1}^{(m)}$ is in $S^{(m)}_{m-1}$, it can be wrote as
   $$\mathcal{C}_{m-1}^{(m)}={\mathcal{L}}_{i}\circledast\hat{\mathcal{C}}_{m-1}^{(m)},$$
   where $\hat{\mathcal{C}}_{m-1}^{(m)}\in\mathcal{D}_{G_{\mathbf{V}^{(m-1)}}}(E_{\mathbf{V}^{(m-1)}})$.
   We have $\mathcal{C}_{m-1}^{(m)}\simeq v^{a^{(m)}_m}\mathbf{1}_{S^{(m)}_{m-1}}=v^{-mN}\mathbf{1}_{S^{(m)}_{m-1}}$. Hence
   $$v^{-(m-1)}\hat{\mathcal{C}}_{m-1}^{(m)}\simeq v^{-mN}\mathbf{1}_{E_{\mathbf{V}^{(m-1)}}},$$
   that is,
   $$\hat{\mathcal{C}}_{m-1}^{(m)}\simeq v^{-mN}v^{m-1}\mathbf{1}_{E_{\mathbf{V}^{(m-1)}}}
   \simeq v^{a^{(m)}_{m-1}}\mathbf{1}_{E_{\mathbf{V}^{(m-1)}}}.$$

   Now, we have constructed $\mathcal{C}_{m-1}^{(m)}$ satisfying the following conditions:
   \begin{enumerate}
   \item[1)] there exists a distinguished triangle
    \begin{displaymath}
       \xymatrix{
      v^{a^{(m)}_{m}}\mathbf{1}_{_iE_{\mathbf{V}^{(p)}}}\ar[r]
      &{\mathcal{C}_{m}^{(m)}}\ar[r]
      &\mathcal{C}_{m-1}^{(m)}\ar[r]
      &;
      }
  \end{displaymath}
   \item[2)] $\mathcal{C}_{m-1}^{(m)}={\mathcal{L}}_{i}\circledast\hat{\mathcal{C}}_{m-1}^{(m)}$,
       where $\hat{\mathcal{C}}_{m-1}^{(m)}\in\mathcal{D}_{G_{\mathbf{V}^{(m-1)}}}(E_{\mathbf{V}^{(m-1)}})$ and
       $\hat{\mathcal{C}}_{m-1}^{(m)}\simeq v^{a^{(m)}_{m-1}}\mathbf{1}_{E_{\mathbf{V}^{(m-1)}}}$.
   \end{enumerate}
\item[\textbf{(3)}] Assume that we have constructed $\mathcal{C}_{p}^{(m)}$ satisfying the following conditions:
   \begin{enumerate}
   \item[1)] there exists a distinguished triangle
    \begin{displaymath}
       \xymatrix{
      v^{a^{(m)}_{p+1}}({\mathcal{L}}_{(m-p-1)i}\circledast\mathbf{1}_{_iE_{\mathbf{V}^{(p+1)}}})\ar[r]
      &{\mathcal{C}_{p+1}^{(m)}}\ar[r]
      &\mathcal{C}_{p}^{(m)}\ar[r]
      &;
      }
  \end{displaymath}
   \item[2)] $\mathcal{C}_{p}^{(m)}={\mathcal{L}}_{(m-p)i}\circledast\hat{\mathcal{C}}_{p}^{(m)}$,
       where $\hat{\mathcal{C}}_{p}^{(m)}\in\mathcal{D}_{G_{\mathbf{V}^{(p)}}}(E_{\mathbf{V}^{(p)}})$ and
       $\hat{\mathcal{C}}_{p}^{(m)}\simeq v^{a^{(m)}_p}\mathbf{1}_{E_{\mathbf{V}^{(p)}}}$.
   \end{enumerate}
   We shall construct $\mathcal{C}_{p-1}^{(m)}$.
    First, there is a distinguished triangle
\begin{displaymath}
       \xymatrix{
      {(j^{(p)}_{p})_!(j^{(p)}_{p})^!(\hat{\mathcal{C}}_{p}^{(m)})}\ar[r]
      &{\hat{\mathcal{C}}_{p}^{(m)}}\ar[r]
      &{(i^{(p)}_{p})_\ast(i^{(p)}_{p})^\ast(\hat{\mathcal{C}}_{p}^{(m)})}\ar[r]
      &.
      }
\end{displaymath}
Hence, we have
\begin{eqnarray}\label{equation:5.2.2}
       \xymatrix{
      {{\mathcal{L}}_{(m-p)i}\circledast(j^{(p)}_{p})_!(j^{(p)}_{p})^!(\hat{\mathcal{C}}_{p}^{(m)})}\ar[r]
      &{{\mathcal{L}}_{(m-p)i}\circledast\hat{\mathcal{C}}_{p}^{(m)}}
      }\nonumber\\
      \xymatrix{
      \ar[r]
      &{{\mathcal{L}}_{(m-p)i}\circledast(i^{(p)}_{p})_\ast(i^{(p)}_{p})^\ast(\hat{\mathcal{C}}_{p}^{(m)})}\ar[r]
      &.
      }
\end{eqnarray}
Since $\hat{\mathcal{C}}_{p}^{(m)}\simeq v^{a^{(m)}_p}\mathbf{1}_{E_{\mathbf{V}^{(p)}}}$,
$$(j^{(p)}_{p})_!(j^{(p)}_{p})^!(\hat{\mathcal{C}}_{p}^{(m)})
\simeq v^{a^{(m)}_p}\mathbf{1}_{_iE_{\mathbf{V}^{(p)}}},$$
and
$$(i^{(p)}_{p})_\ast(i^{(p)}_{p})^\ast(\hat{\mathcal{C}}_{p}^{(m)})
\simeq v^{a^{(m)}_p}\mathbf{1}_{S^{(p)}_{p-1}}.$$
Let $\mathcal{C}_{p-1}^{(m)}={\mathcal{L}}_{(m-p)i}\circledast(i^{(p)}_{p})_\ast(i^{(p)}_{p})^\ast(\hat{\mathcal{C}}_{p}^{(m)})$.
By (\ref{equation:5.2.2}), there exists a distinguished triangle
   \begin{displaymath}
       \xymatrix{
      {v^{a^{(m)}_{p}}({\mathcal{L}}_{(m-p)i}\circledast\mathbf{1}_{_iE_{\mathbf{V}^{(p)}}}})\ar[r]
      &{\mathcal{C}_{p}^{(m)}}\ar[r]
      &\mathcal{C}_{p-1}^{(m)}\ar[r]&.
      }
\end{displaymath}
Since the support of $\mathcal{C}_{p-1}^{(m)}$ is in $S^{(m)}_{p-1}$, it can be wrote as
   $$\mathcal{C}_{p-1}^{(m)}={\mathcal{L}}_{(m-p+1)i}\circledast\hat{\mathcal{C}}_{p-1}^{(m)},$$
   where $\hat{\mathcal{C}}_{p-1}^{(m)}\in\mathcal{D}_{G_{\mathbf{V}^{(p-1)}}}(E_{\mathbf{V}^{(p-1)}})$.
   Since
   $$(i^{(p)}_{p})_\ast(i^{(p)}_{p})^\ast(\hat{\mathcal{C}}_{p}^{(m)})
   \simeq v^{a^{(m)}_p}\mathbf{1}_{S^{(p)}_{p-1}},$$
   we have
   $$\mathcal{C}_{p-1}^{(m)}={{\mathcal{L}}_{(m-p)i}\circledast(i^{(p)}_{p})_\ast(i^{(p)}_{p})^\ast(\hat{\mathcal{C}}_{p}^{(m)})}
   \simeq v^{-(m-p)p}v^{a^{(m)}_p}\mathbf{1}_{S^{(m)}_{p-1}}\simeq v^{-mN}\mathbf{1}_{S^{(m)}_{p-1}}.$$
   Hence
$$v^{-(m-p+1)(p-1)}\hat{\mathcal{C}}_{p-1}^{(m)}=v^{-mN}\mathbf{1}_{E_{\mathbf{V}^{(p-1)}}},$$
that is
$$\hat{\mathcal{C}}_{p-1}^{(m)}\simeq v^{-mN}v^{(m-p+1)(p-1)}\mathbf{1}_{E_{\mathbf{V}^{(p-1)}}}\simeq v^{a^{(m)}_{p-1}}\mathbf{1}_{E_{\mathbf{V}^{(p-1)}}}.$$

Now, we have constructed $\mathcal{C}_{p-1}^{(m)}$ satisfying the following conditions:
   \begin{enumerate}
   \item[1)] there exists a distinguished triangle
    \begin{displaymath}
       \xymatrix{
      v^{a^{(m)}_{p}}({\mathcal{L}}_{(m-p)i}\circledast\mathbf{1}_{_iE_{\mathbf{V}^{(p)}}})\ar[r]
      &{\mathcal{C}_{p}^{(m)}}\ar[r]
      &\mathcal{C}_{p-1}^{(m)}\ar[r]
      &;
      }
  \end{displaymath}
   \item[2)] $\mathcal{C}_{p-1}^{(m)}={\mathcal{L}}_{(m-p+1)i}\circledast\hat{\mathcal{C}}_{p-1}^{(m)}$,
       where $\hat{\mathcal{C}}_{p-1}^{(m)}\in\mathcal{D}_{G_{\mathbf{V}^{(p-1)}}}(E_{\mathbf{V}^{(p-1)}})$ and
       $\hat{\mathcal{C}}_{p-1}^{(m)}\simeq v^{a^{(m)}_{p-1}}\mathbf{1}_{E_{\mathbf{V}^{(p-1)}}}$.
   \end{enumerate}

By induction, the proof is finished.

\end{proof}

In Section \ref{subsection:4.1}, we have $\chi:K(\mathcal{Q})\rightarrow\mathcal{F}$.
In this section, we identify the Lusztig's algebra $\mathbf{f}$ with the corresponding composition subalgebra $\mathcal{F}$.

Lusztig proved the following theorem.

\begin{theorem}[\cite{Lusztig_Quivers_perverse_sheaves_and_the_quantized_enveloping_algebras}]\label{theorem:5.1}
$\chi(I^{(m)}_p)=\theta_i^{(m-p)}\theta_j\theta_i^{(p)}$ for each $m\geq p\in\mathbb{N}$.
\end{theorem}

By Proposition \ref{theorem:5.3} and Theorem \ref{theorem:5.1}, we have the following corollary.

\begin{corollary}\label{corollary:5.1}
We have the following formula in $\mathbf{f}$
$$\theta_j\theta_i^{(m)}=\sum_{p=0}^m v^{b^{(m)}_p}\theta_i^{(m-p)}\chi(\mathcal{E}^{(p)}),$$
where $b^{(m)}_p=(p-N)(m-p)$.
\end{corollary}

\begin{proof}[\bf{Proof}]
By Proposition \ref{theorem:5.3} and Theorem \ref{theorem:5.1}, we have
$$\theta_j\theta_i^{(m)}=\sum_{p=0}^m v^{a^{(m)}_p}\theta_i^{(m-p)}\chi(\mathbf{1}_{_iE_{\mathbf{V}^{(p)}}})
=\sum_{p=0}^m v^{a^{(m)}_p}v^{pN}\theta_i^{(m-p)}\chi(\mathcal{E}^{(p)}).$$
Since $a^{(m)}_p+pN=b^{(m)}_p$,
we have
$$\theta_j\theta_i^{(m)}=\sum_{p=0}^m v^{b^{(m)}_p}\theta_i^{(m-p)}\chi(\mathcal{E}^{(p)}).$$
\end{proof}

We shall use Lemma \ref{lemma:5.2} and Proposition \ref{theorem:5.3} to prove Theorem \ref{theorem:5.4} by induction.

\begin{proof}[\bf{Proof of Theorem \ref{theorem:5.4}}]

We shall prove this result by induction on $m$.

\item[\textbf{(1)}] For $m=0$, $\mathcal{E}^{(0)}=I^{(0)}_{0}$. It is clear that $\mathcal{E}^{(0)}$ is with Property \textbf{A}(0).
\item[\textbf{(2)}] For $m=1$, by Proposition \ref{theorem:5.3}, there exists a distinguished triangle
      \begin{displaymath}
       \xymatrix{
      v^{-N}\mathbf{1}_{_iE_{\mathbf{V}^{(1)}}}\ar[r]
      &{\mathcal{C}_1^{(1)}}\ar[r]
      &\mathcal{C}_{0}^{(1)}\ar[r]
      &,
      }
      \end{displaymath}
      where $\mathcal{C}_1^{(1)}=I^{(1)}_1$ and $\mathcal{C}^{(1)}_{0}=v^{-N}({\mathcal{L}}_{i}\circledast\mathbf{1}_{_iE_{\mathbf{V}^{(0)}}})$.
      Since $\mathcal{E}^{(0)}=I^{(0)}_{0}$, $$\mathcal{C}^{(1)}_{0}=
      v^{-N}({\mathcal{L}}_{i}\circledast\mathbf{1}_{_iE_{\mathbf{V}^{(0)}}})
      ={\mathcal{L}}_{i}\circledast\mathcal{E}^{(0)}$$
      is the direct sum of some semisimple perverse sheaves of the form $I^{(1)}_{p'}[l]$. Hence, $\mathcal{E}^{(1)}=v^{-N}\mathbf{1}_{_iE_{\mathbf{V}^{(1)}}}$ is with Property \textbf{A}(1).
\item[\textbf{(3)}] Assume the $\mathcal{E}^{(k)}$ is with Property \textbf{A}(k) for all $k<m$. Let us prove $\mathcal{E}^{(m)}$ is with Property \textbf{A}(m).

    For any $k<m$, there exists $s_{k}\in\mathbb{N}$. For each $s_k\geq p\in\mathbb{N}$, there exists $\mathcal{E}_p^{(k)}\in\mathcal{D}_{G_{\mathbf{V}^{(k}}}(E_{\mathbf{V}^{(k)}})$ such that
    \begin{enumerate}
    \item[1)]$\mathcal{E}_{s_k}^{(k)}=\mathcal{E}^{(k)}$ and $\mathcal{E}_{0}^{(k)}$ is the direct sum of some semisimple perverse sheaves of the form $I^{(k)}_{p'}[l]$;
    \item[2)]for each $p\geq1$, there exists a distinguished triangle
    \begin{displaymath}
    \xymatrix{
    \mathcal{E}^{(k)}_{p}\ar[r]&{\mathcal{G}^{(k)}_{p}}\ar[r]&\mathcal{E}^{(k)}_{p-1}\ar[r]&,
    }
    \end{displaymath}
    where $\mathcal{G}^{(k)}_{p}$ is the direct sum of some semisimple perverse sheaves of the form $I^{(k)}_{p'}[l]$.
    \end{enumerate}
    Hence, we have the following distinguished triangle for each $p\geq1$
    \begin{displaymath}
    \xymatrix{
    {\mathcal{L}}_{(m-k)i}\circledast\mathcal{E}^{(k)}_{p}\ar[r]
    &{\mathcal{L}}_{(m-k)i}\circledast{\mathcal{G}^{(k)}_{p}}\ar[r]
    &{\mathcal{L}}_{(m-k)i}\circledast\mathcal{E}^{(k)}_{p-1}\ar[r]&.
    }
    \end{displaymath}
    Denote $\tilde{\mathcal{E}}^{(k)}_{p}={\mathcal{L}}_{(m-k)i}\circledast\mathcal{E}^{(k)}_{p}$ and $\tilde{\mathcal{G}}^{(k)}_{p}={{\mathcal{L}}_{(m-k)i}\circledast\mathcal{G}^{(k)}_{p}}$.
    Then, we have
    \begin{displaymath}
    \xymatrix{
    \tilde{\mathcal{E}}^{(k)}_{p}\ar[r]
    &{\tilde{\mathcal{G}}^{(k)}_{p}}\ar[r]
    &\tilde{\mathcal{E}}^{(k)}_{p-1}\ar[r]&.
    }
    \end{displaymath}
    Because $\tilde{\mathcal{E}}^{(k)}_{0}$ and $\tilde{\mathcal{G}}^{(k)}_{p}$ are the direct sums of some semisimple perverse sheaves of the form $I^{(m)}_{p'}[l]$, $\tilde{\mathcal{E}}^{(k)}_{k}$ is with Property \textbf{A}(m). Since
    $$\tilde{\mathcal{E}}^{(k)}_{k}={\mathcal{L}}_{(m-k)i}\circledast\mathcal{E}^{(k)}_{k}
    =v^{-kN}({\mathcal{L}}_{(m-k)i}\circledast\mathbf{1}_{_iE_{\mathbf{V}^{(k)}}}),$$
    ${\mathcal{L}}_{(m-k)i}\circledast\mathbf{1}_{_iE_{\mathbf{V}^{(k)}}}$ is with Property \textbf{A}(m).

    By Proposition \ref{theorem:5.3}, for each $m\geq k\in\mathbb{N}$, there exists $\mathcal{C}_k^{(m)}\in\mathcal{D}_{G_{\mathbf{V}^{(m)}}}(E_{\mathbf{V}^{(m)}})$ such that
    \begin{enumerate}
    \item[1)]$\mathcal{C}_m^{(m)}=I^{(m)}_m$ and $\mathcal{C}^{(m)}_{0}=v^{-mN}({\mathcal{L}}_{mi}\circledast\mathbf{1}_{_iE_{\mathbf{V}^{(0)}}})$;
    \item[2)]for each $k\geq1$, there exists a distinguished triangle
    \begin{displaymath}
       \xymatrix{
      v^{a^{(m)}_k}({\mathcal{L}}_{(m-k)i}\circledast\mathbf{1}_{_iE_{\mathbf{V}^{(k)}}})\ar[r]
      &{\mathcal{C}_k^{(m)}}\ar[r]
      &\mathcal{C}_{k-1}^{(m)}\ar[r]
      &.
      }
    \end{displaymath}
    \end{enumerate}
    We have proved that $\mathcal{C}^{(m)}_{0}$ and ${\mathcal{L}}_{(m-k)i}\circledast\mathbf{1}_{_iE_{\mathbf{V}^{(k)}}}$ ($1\leq k\leq m-1$)
    are with Property \textbf{A}(m). Hence, by Lemma \ref{lemma:5.2}, $\mathcal{C}^{(m)}_{m-1}$ is with Property \textbf{A}(m).
    At last, by Lemma \ref{lemma:5.2} and the distinguished triangle
    \begin{displaymath}
      \xymatrix{
      \mathcal{C}_{m-1}^{(m)}[-1]\ar[r]
      &v^{-mN}\mathbf{1}_{_iE_{\mathbf{V}^{(p)}}}\ar[r]
      &I^{(m)}_m\ar[r]
      &,
      }
    \end{displaymath}
    $\mathcal{E}^{(m)}=v^{-mN}\mathbf{1}_{_iE_{\mathbf{V}^{(p)}}}$ is with Property \textbf{A}(m).

By induction, the proof is finished.

\end{proof}

As a corollary of Theorem \ref{theorem:5.4}, we have

\begin{corollary}\label{corollary:5.2}
For each $N\geq m\in\mathbb{N}$, we have the following formula
$$\chi(\mathcal{E}^{(m)})=\sum_{p=0}^{m}(-1)^pv^{-p(1+N-m)}\theta_i^{(p)}\theta_j\theta_i^{(m-p)}=f(i,j;m).$$
\end{corollary}

\begin{proof}[\bf{Proof}]

By Theorem \ref{theorem:5.4}, we have
$$\chi(\mathcal{E}^{(m)})=\sum_{p=0}^{m}c^{(m)}_p\theta_i^{(p)}\theta_j\theta_i^{(m-p)}.$$
We shall prove that $c^{(m)}_p=(-1)^pv^{-p(1+N-m)}$ ($0\leq p\leq m$) by induction on $m$.

\item[\textbf{(1)}] For $m=0$, by Corollary \ref{corollary:5.1},
$$\theta_j=\chi(\mathcal{E}^{(0)}).$$
That is $c^{(0)}_0=1$. Hence, the corollary is true in this case.

\item[\textbf{(2)}] Assume that $c^{(k)}_p=(-1)^pv^{-p(1+N-k)}$ ($0\leq p\leq k$) for any $k<m$.
We shall prove that $c^{(m)}_q=(-1)^qv^{-q(1+N-m)}$ ($0\leq q\leq m$).

By Corollary \ref{corollary:5.1},
\begin{eqnarray*}
\theta_j\theta_i^{(m)}&=&\sum_{k=0}^m v^{b^{(m)}_k}\theta_i^{(m-k)}\chi(\mathcal{E}^{(k)})\\
&=&\sum_{k=0}^{m-1} v^{b^{(m)}_k}\theta_i^{(m-k)}\chi(\mathcal{E}^{(k)})
+v^{b^{(m)}_m}\chi(\mathcal{E}^{(m)})\\
&=&\sum_{k=0}^{m-1} v^{b^{(m)}_k}\theta_i^{(m-k)}\chi(\mathcal{E}^{(k)})
+\chi(\mathcal{E}^{(m)}).
\end{eqnarray*}
Hence,
\begin{eqnarray*}
\chi(\mathcal{E}^{(m)})
&=&\theta_j\theta_i^{(m)}
-\sum_{k=0}^{m-1} v^{b^{(m)}_k}\theta_i^{(m-k)}\chi(\mathcal{E}^{(k)})\\
&=&\theta_j\theta_i^{(m)}
-\sum_{k=0}^{m-1} v^{b^{(m)}_k}\theta_i^{(m-k)}
\sum_{p=0}^{k}c^{(k)}_p\theta_i^{(p)}\theta_j\theta_i^{(k-p)}\\
&=&\theta_j\theta_i^{(m)}
-\sum_{k=0}^{m-1}\sum_{p=0}^{k}v^{b^{(m)}_k}c^{(k)}_p\frac{[m-k+p]_v!}{[m-k]_v![p]_v!}
\theta_i^{(m-k+p)}\theta_j\theta_i^{(k-p)}.
\end{eqnarray*}
For any $q\geq1$,
\begin{eqnarray*}
c^{(m)}_q
&=&-\sum_{k=m-q}^{m-1}v^{b^{(m)}_k}c^{(k)}_{q+k-m}\frac{[q]_v!}{[m-k]_v![q+k-m]_v!}\\
&=&-\sum_{k=0}^{q-1}v^{b^{(m)}_{k+m-q}}c^{(k+m-q)}_k\frac{[q]_v!}{[q-k]_v![k]_v!}.
\end{eqnarray*}
By the induction hypothesis,
\begin{eqnarray*}
c^{(m)}_q
&=&-\sum_{k=0}^{q-1}v^{b^{(m)}_{k+m-q}}c^{(k+m-q)}_k\frac{[q]_v!}{[q-k]_v![k]_v!}\\
&=&-\sum_{k=0}^{q-1}(-1)^{k}v^{(k+m-q-N)(q-k)}v^{-k(1+N-k-m+q)}\frac{[q]_v!}{[q-k]_v![k]_v!}\\
&=&-v^{q(m-q-N)}\sum_{k=0}^{q-1}(-1)^{k}v^{k(q-1)}\frac{[q]_v!}{[q-k]_v![k]_v!}\\
&=&-v^{q(m-q-N)}\sum_{k=0}^{q}(-1)^{k}v^{k(q-1)}\frac{[q]_v!}{[q-k]_v![k]_v!}
+v^{q(m-q-N)}(-1)^{q}v^{q(q-1)}\\
&=&v^{q(m-q-N)}(-1)^{q}v^{q(q-1)}=(-1)^{q}v^{-q(1+N-m)}
\end{eqnarray*}
Note that
\begin{displaymath}
c^{(m)}_0=1=(-1)^0v^{-0(1+N-m)}.
\end{displaymath}
Hence, $c^{(m)}_q=(-1)^qv^{-q(1+N-m)}$ for any $0\leq q\leq m$.

By induction, for each $N\geq m\in\mathbb{N}$,  $c^{(m)}_p=(-1)^pv^{-p(1+N-m)}$ ($0\leq p\leq m$) and
$$\chi(\mathcal{E}^{(m)})=\sum_{p=0}^{m}(-1)^pv^{-p(1+N-m)}\theta_i^{(p)}\theta_j\theta_i^{(m-p)}.$$

\end{proof}

\subsection{The formulas of Lusztig's symmetries}\label{subsection:5.4}

In this section, we shall give a new proof of Proposition \ref{proposition:5.1}.

Consider the following quiver
\begin{displaymath}\label{equation:5.2.3}
Q: \xymatrix@=66pt{i&j\ar@/^/[l]_{.}\ar@/_/[l]^{.}\\}
\end{displaymath}
with vertex set $I=\{i,j\}$ and $N$ arrows from $j$ to $i$. Let $Q'=\sigma_i Q$ be the quiver by reversing the directions of all arrows
$$Q': \xymatrix@=66pt{i\ar@/^/[r]_{.}\ar@/_/[r]^{.}&j\\}$$

Let $m$ be a non-negative integer such that $m\leq N$ and $m'=N-m$. Let $\nu=mi+j\in\mathbb{N}I$ and $\nu'=s_i\nu=m'i+j\in\mathbb{N}I$. Fix two $I$-graded $\mathbb{K}$-vector spaces $\mathbf{V}$ and $\mathbf{V}'$ such that $\underline{\dim}\mathbf{V}=\nu$ and $\underline{\dim}\mathbf{V}'=\nu'$.

Denote by $\mathbf{1}_{_iE_{\mathbf{V},Q}}\in\mathcal{D}_{G_{\mathbf{V}}}(_iE_{\mathbf{V},Q})$ the constant sheaf on $_iE_{\mathbf{V},Q}$ and $\mathbf{1}_{^iE_{\mathbf{V}',Q'}}\in\mathcal{D}_{G_{\mathbf{V}'}}(^iE_{\mathbf{V}',Q'})$ the constant sheaf on $^iE_{\mathbf{V}',Q'}$. For convenience, denote $_iE_{\mathbf{V},Q}$ (resp. $^iE_{\mathbf{V}',Q'}$) by $_iE_{\mathbf{V}}$ (resp. $^iE_{\mathbf{V}'}$) and
$\mathbf{1}_{_iE_{\mathbf{V},Q}}$ (resp. $\mathbf{1}_{^iE_{\mathbf{V}',Q'}}$)
by $\mathbf{1}_{_iE_{\mathbf{V}}}$ (resp. $\mathbf{1}_{^iE_{\mathbf{V}'}}$).

Denote
$$\mathcal{E}^{(m)}=j_{\mathbf{V}!}(v^{-mN}\mathbf{1}_{_iE_{\mathbf{V}}})\in\mathcal{D}_{G_{\mathbf{V}}}(E_{\mathbf{V}})$$
and
$${\mathcal{E}'}^{(m')}=j_{\mathbf{V}'!}(v^{-m'N}\mathbf{1}_{^iE_{\mathbf{V}'}})\in\mathcal{D}_{G_{\mathbf{V}'}}(E_{\mathbf{V}'}).$$

In Section \ref{subsection:3.4}, we give the following geometric realization of the Lusztig's symmetry $T_i$:
$$\tilde{\omega}_i:\mathcal{D}_{G_{\mathbf{V}}}(_iE_{\mathbf{V}})\rightarrow\mathcal{D}_{G_{\mathbf{V}'}}(^iE_{\mathbf{V}'}).$$

\begin{proposition}\label{proposition:5.2}
For any $N\geq m\in\mathbb{N}$,
$\tilde{\omega}_i(v^{-mN}\mathbf{1}_{_iE_{\mathbf{V}}})=v^{-m'N}\mathbf{1}_{^iE_{\mathbf{V}'}}$.
\end{proposition}

\begin{proof}[\bf{Proof}]
By the definitions of $\alpha$ and $\beta$ in the diagram (\ref{equation:3.4.2}) of Section \ref{subsection:3.4},
$$\alpha^{\ast}(\mathbf{1}_{_iE_{\mathbf{V}}})=\mathbf{1}_{Z_{\mathbf{V}\mathbf{V}'}}=\beta^{\ast}(\mathbf{1}_{^iE_{\mathbf{V}'}}).$$
Hence $$\tilde{\omega}_i(\mathbf{1}_{_iE_{\mathbf{V}}})=v^{(m-m')N}\mathbf{1}_{^iE_{\mathbf{V}'}}.$$
That is
$$\tilde{\omega}_i(v^{-mN}\mathbf{1}_{_iE_{\mathbf{V}}})=v^{-m'N}\mathbf{1}_{^iE_{\mathbf{V}'}}.$$

\end{proof}

Corollary \ref{corollary:5.2} implies
$\chi(\mathcal{E}^{(m)})=f(i,j;m)$. Similarly, we have $\chi(\mathcal{E}'^{(m')})=f'(i,j;m')$. Hence,
Proposition \ref{proposition:5.2} implies Proposition \ref{proposition:5.1}.

\bibliography{mybibfile}

\end{document}